\documentclass[11pt]{amsart}

\usepackage[top=1.5in, left=1.5in, bottom=1.5in, right=1.5in]{geometry}

\usepackage{amsmath}
\usepackage{amsthm}
\usepackage{amssymb}
\usepackage{graphicx}
\usepackage{mathrsfs}
\usepackage{hyperref} 
\usepackage[hypcap]{caption}

\usepackage[shortlabels]{enumitem}
\usepackage{chngcntr}

\newtheorem{mainthm}{Theorem}
\newtheorem{maincor}[mainthm]{Corollary}

\newtheorem{thm}{Theorem}[section]

\newtheorem{lem}[thm]{Lemma}
\newtheorem{prop}[thm]{Proposition}

\newtheorem*{thmnon}{Theorem}

\theoremstyle{definition}
\newtheorem{defn}{Definition}

\theoremstyle{remark}

\newtheorem*{claim}{Claim}

\newcommand{\Cc}{\mathcal C}

\newcommand{\Bb}{\mathcal B}

\newcommand{\Om}{\Omega}
\newcommand{\GB}{\ensuremath{(G,\Bb)}}
\newcommand{\eq}{{=}}

\newcommand{\ie}{\textit{i.e.}}
\newcommand{\Ie}{\textit{I.e.}}
\newcommand{\iiff}{if and only if }
\newcommand{\wolog}{without loss of generality}

\DeclareMathOperator{\cl}{cl}  
\DeclareMathOperator{\rank}{rank}  
\newcommand{\si}{\mathop{\mathrm{si}}}
\newcommand{\re}{\mathop{\mathrm{re}}}
\newcommand{\br}[1]{\left( #1 \right)}

\newcommand{\comout}[1]{}

\title[Almost balanced representations]{Almost balanced biased graph representations of frame matroids}
\author[DeVos \and Funk]{Matt DeVos \and Daryl Funk}
\date{\today}

\begin{document}

\maketitle

\begin{abstract}
Given a 3-connected biased graph $\Om$ with a balancing vertex, and with frame matroid $F(\Om)$ nongraphic and 3-connected, we determine all biased graphs $\Om'$ with $F(\Om') = F(\Om)$.  
As a consequence, we show that if $M$ is a 4-connected nongraphic frame matroid represented by a biased graph $\Om$ having a balancing vertex, then $\Om$ essentially uniquely represents $M$. 
More precisely, all biased graphs representing $M$ are obtained from $\Om$ by replacing a subset of the edges incident to its unique balancing vertex with unbalanced loops.  
\end{abstract}

A \emph{frame} for a matroid is a basis $B$ with the property that every element of the matroid is spanned by at most two elements of $B$.  
If a matroid $M$ may be extended so that it contains such a basis, then $M$ is a \emph{frame matroid}.  
Subclasses of frame matroids have recently been seen to play a fundamental role in matroid structure theory \cite{solvingrotasconj}.  
Frame matroids are a natural generalization of graphic matroids.  
Indeed, the cycle matroid $M(G)$ of a graph $G = (V,E)$ is naturally extended by adding $V$ as a basis, and declaring each edge to be minimally spanned by its endpoints.  
Zaslavsky has shown that the class of frame matroids is precisely that of matroids arising from \emph{biased graphs} \cite{MR1273951}.  
A {biased graph} consists of a pair $(G, \mathcal{B})$, where $G$ is a graph and $\mathcal{B}$ is a collection of cycles of $G$, called \emph{balanced}, such that no theta subgraph contains exactly two balanced cycles; a \emph{theta} graph consists of a pair of distinct vertices and three internally disjoint paths between them.  
Every biased graph $\GB$ gives rise to a frame matroid, which we denote $F\GB$, and for every frame matroid $M$ there is at least one biased graph $\GB$ with $F\GB$ isomorphic to $M$.  
We say such a biased graph $\GB$ \emph{represents} the frame matroid $M$, and write $M = F\GB$.

\subsection*{Given a frame matroid $M$, which biased graphs represent $M$?} 

A well-known result of Whitney says that if a graph $G$ has no loop and is 3-connected, then the cycle matroid $M(G)$ is uniquely represented by $G$.  
The analogous starting point for the study of representations of frame matroids by biased graphs is the following result of Slilaty.  
(The connectivity of a biased graph $\GB$ is that of $G$.)  

\begin{thmnon}[Slilaty \cite{MR2245651}] \label{thm:Silaty_unique_biased_graph_rep}
Let $\GB$ be a 3-connected biased graph with no balanced loop.  
If $\GB$ contains three disjoint unbalanced cycles, at most one of which is a loop, then $\GB$ uniquely represents $F\GB$.  
\end{thmnon}

Little else is known about representations of general frame matroids by biased graphs.  
Those biased graphs representing graphic matroids are known \cite{RongPaper}, and there have been studies on representations of subclasses (for example, \cite{MR3032922, MR1892972, MR2344133, slilaty:Decompositionssignedgraphicmatroids, MR2394450}).  
In this paper, we determine all biased graph representations of frame matroids that arise from biased graphs having a special structure.  
A \emph{balancing vertex} in a biased graph is a vertex whose deletion destroys all unbalanced cycles.  
We say a biased graph is \emph{almost balanced} if after removing unbalanced loops it has a balancing vertex. 
Given a 3-connected almost balanced biased graph $\Om=\GB$, we determine all other biased graphs representing $F(\Om)$.  
This is the content of Theorem \ref{mainthm:main_bal_vertex_rep}, our main result. 
The technical terms \emph{roll-up}, \emph{$H$-reduction}, and \emph{$H$-enlargement} will be explained in Section \ref{sec:operations}. 

\begin{mainthm} \label{mainthm:main_bal_vertex_rep}
Let $\Om$ be a 3-connected almost balanced biased graph with no balanced loop and with $F(\Om)$ nongraphic.  
Suppose $\Om'$ is a biased graph with $F(\Om') = F(\Om)$.  
Then either $\Om'$ is a roll-up of $\Om$, or 
there is a subgraph $H$ of $\Om$ and a pair of biased graphs $\Psi$ and $\Psi'$ on at most six vertices, with $F(\Psi) = F(\Psi')$, such that 
$\Psi$ is an $H$-reduction of $\Om$ and $\Om'$ is an $H$-enlargement of $\Psi'$.  
\end{mainthm} 

Theorem \ref{mainthm:main_bal_vertex_rep} is interesting mainly for its following two corollaries. 

\begin{maincor} \label{maincor:notsomanybiasedgraphreps}
Let $\Om$ be a 3-connected almost balanced biased graph with no balanced loop and with $F(\Om)$ nongraphic.  
Up to roll-ups the number of biased graph representations of $F(\Om)$ is at most $27$.  
\end{maincor}

\begin{maincor} \label{maincor:fourconnbalvertexrepunique}
Let $\Om$ be an almost balanced biased graph with $F(\Om)$ nongraphic and 4-connected.  
Then up to roll-ups, $\Om$ uniquely represents $F(\Om)$.  
\end{maincor}

A simple example illustrates the operation of a \emph{roll-up}, and its necessity.  
Let $\Om$ be the biased graph obtained from an $n$-vertex graph $H$ by adding a vertex $u$ together with $k$ edges between $u$ and each vertex of $H$, with balanced cycles just the cycles of $H$.  
Then $u$ is a balancing vertex of $\Om$, and $F(\Om)$ has rank $n+1$.  
A \emph{roll-up} of an edge $e=uv$ is the operation of redefining its incidence so that $e$ becomes an unbalanced loop incident to its endpoint $v$.  
In this example, every biased graph obtained by a roll-up of an edge incident to $u$ also represents $F(\Om)$.  
Hence there are at least $kn$ different representations for $F(\Om)$.  
Moreover, as long as $H$ is simple and connected, $F(\Om)$ is 3-connected.  
Since $k$ may be taken arbitrarily large, this shows that, for fixed $r$, there are 3-connected rank $r$ frame matroids represented by a biased graph with a balancing vertex, having arbitrarily many other biased graph representations.  

The situation appears better if we focus on frames rather than representations.  
Let $M$ be a frame matroid.  
Then $M$ has a frame $B$, and a biased graph $\GB$ representing $M$ may be constructed as follows \cite{MR1273951}.  
By adding elements in parallel if necessary, we may assume $B$ is disjoint from $E(M)$.  
Set $V(G)=B$ and for each element $e$ of $M$, put $e$ as an edge with endpoints $x,y$ in $G$ if $e$ is minimally spanned by $x,y \in B$; place $e$ as an unbalanced loop incident to $x$ if $e$ is in parallel with $x$, and if $e$ is a loop of $M$ place $e$ as a balanced loop incident to an arbitrary vertex. 
Define $\Bb$ to be those cycles of $G$ whose edge set is a circuit of $M$.  
Thus different biased graph representations of $M$ arise from different choices of frames for $M$. 
Conversely, given a biased graph representation, the vertices of the graph provide a frame for $M$.  
In the example above, a roll-up corresponds geometrically to the fact that sliding $e \in \text{span}\{u,v\}$ along the line spanned by $u$ and $v$ until $e$ sits in parallel with $v$ does not change the matroid $F(\Om)$. 

Formally, let us say that two frames $B_1$ and $B_2$ for $M$ are \emph{the same} if their elements can be labelled $B_1 = \{b_1, b_2, \ldots, b_n\}$ and $B_2 = \{c_1, c_2, \ldots, c_n\}$ respectively, so that for every $e \in E(M)$, $e$ is in the span of $\{b_i, b_j\}$ if and only if $e$ is in the span of $\{c_i, c_j\}$. Note that this permits an element minimally spanned by two elements of $B_1$ to be in parallel with an element of $B_2$. Otherwise the two frames are \emph{different}.  
Roll-ups arise as a collection of biased graph representations all sharing the same frame for $M$, where certain elements of $M$ may be placed in parallel with certain elements of the frame (details are provided in Sections \ref{sec:rerouting} and \ref{sec:operations}).  
Corollaries \ref{maincor:notsomanybiasedgraphreps} and \ref{maincor:fourconnbalvertexrepunique} may therefore equivalently be stated as follows.  

\setcounter{mainthm}{1}
\begin{maincor} \label{maincor:notsomanybiasedgraphrepsframestatement}
Let $\Om$ be a 3-connected almost balanced biased graph with $F(\Om)$ nongraphic.  
There are at most 27 different frames for $F(\Om)$.  
\end{maincor}

\begin{maincor} \label{maincor:fourconnbalvertexrepuniqueframestatement} 
Let $\Om$ be an almost balanced biased graph with $F(\Om)$ nongraphic and 4-connected.  
The vertex set of $\Om$ provides the unique frame for $F(\Om)$.  
\end{maincor}

We will show that Corollaries \ref{maincor:notsomanybiasedgraphreps} and \ref{maincor:fourconnbalvertexrepunique} follow from Theorem \ref{mainthm:main_bal_vertex_rep} in Section \ref{sec:proofsofcorollaries}, after explaining the required preliminary concepts.  
For basic concepts in matroid theory, we refer to Oxley's standard text \cite{oxley:mt2011}.

\subsection*{Remark} 
Frame matroids having an almost balanced biased graph representation are a rather special class of frame matroids.  
It is one of six classes we have identified as being one vertex away from being graphic, in the following sense.  

Our main tool in the proof of Theorem \ref{mainthm:main_bal_vertex_rep} is the notion of a \emph{committed} vertex (Definition \ref{defn:Committedvertex}).  
This is a vertex whose set of incident edges forms a cocircuit whose complementary hyperplane is connected and nongraphic.  
These edges must, in any biased graph representation of the matroid, all remain incident to a common vertex (Proposition  \ref{prop:comp_cocircuit_connnonbinhyperplane}).  
This enables us to show that the biased graphs under consideration have large subgraphs that must appear essentially unchanged in any biased graph representation of the matroid (Lemmas \ref{lem:possible_committed_lobe_reps} and \ref{lem:possible_committed_lobe_reps_pinched}).  
Alternate biased graphs representing the same frame matroid are possible when a biased graph has vertices that are not committed.  
When connectivity is high enough, the deletion of an uncommitted vertex leaves a connected graphic hyperplane.  
In \cite{RongPaper} we have characterised the biased graphs representing a graphic matroid with a list of six families of biased graphs.  

This suggests the following strategy for determining the biased graphs representing a frame matroid $M$.  
Let $\Om$ be a biased graph representation for $M$.  
If all vertices of $\Om$ are committed, then $\Om$ uniquely represents $M$. 
Otherwise, there is a vertex whose deletion leaves a biased graph in one of our six graphic families.  
We know all representations of these, so all that is required is to determine which vertices are committed and, for those edges incident to an uncommitted vertex, which new incidences are permitted.  
The family of balanced biased graphs (that is, biased graphs in which all cycles are balanced) is the simplest of our six families of biased graphs with graphic frame matroids exhibited in \cite{RongPaper}.  
The current paper tackles the case that upon deletion of an uncommitted vertex, the matroid on the remaining elements is graphic because the resulting biased graph is balanced.  

We can hope that this is the first step in characterising representations of all frame matroids, or at least those whose connectivity is not too low.  
Since the remaining five families of biased graphs with graphic frame matroids to be considered each have large balanced subgraphs containing many committed vertices, the approach we develop here seems promising.

\section{The structure of biases in biased graphs with a balancing vertex} \label{sec:rerouting}

To see the connection between abstract frame matroids and biased graphs, let $M$ be a frame matroid on ground set $E$, with frame $B$.  
By adding elements in parallel if necessary, we may assume $B \cap E = \emptyset$.  
Hence there is a matroid $N$ with $M = N \setminus B$ where $B$ is a basis for $N$ and every element $e \in E$ is spanned by a pair of elements in $B$.  
Let $G$ be a graph with vertex set $B$ and edge set $E$, in which $e$ is a loop with endpoint $v$ if $e$ is in parallel with $v \in B$, $e$ is a loop incident to any vertex if $e$ is a loop of $M$, and otherwise $e$ is an edge with endpoints $u,v \in B$ if $e \in \cl\{u,v\}$.  
Setting $\Bb = \{ C : C$ is a cycle for which $E(C)$ is a circuit of $M \}$ (where a loop is a cycle of length 1) yields a biased graph $\GB$.  
The circuits of $M$ are precisely those sets of edges of $G$ inducing one of: 
(1)  a balanced cycle, 
(2)  two edge-disjoint unbalanced cycles intersecting in just one vertex, 
(3)  two vertex-disjoint unbalanced cycles along with a path connecting them, or 
(4)  a theta subgraph in which all three cycles are unbalanced \cite{MR1273951}.  
We call a subgraph as in (2) or (3) a \emph{pair of handcuffs}, \emph{tight} or \emph{loose}, respectively.  
Conversely, given a graph $G$ together with a collection of cycles $\Bb$ obeying the \emph{theta property} --- \ie, no theta subgraph contains exactly two cycles in $\Bb$ --- there is a frame matroid $F\GB$ arising from $\GB$ defined by taking as circuits precisely those edge sets of $G$ that form balanced cycles, pairs of handcuffs, and theta subgraphs having all three cycles unbalanced.  
From this it is easy to see that the rank function $r$ of a frame matroid represented by the biased graph $\GB$ is $r(X) = |V(X)| - b(X)$, where $V(X)$ denotes the set of vertices incident with an edge in $X$, and $b(X)$ is the number of balanced components of the biased subgraph induced by $X \subseteq E(G)$.  

Since in general an abstract frame matroid $M$ may have more than one frame $B$, and the construction above of a biased graph representing $M$ depends upon the choice of $B$, we see that there may be many different biased graphs representing $M$.

The membership or non-membership of a cycle in $\Bb$ is its \emph{bias}; cycles not in $\Bb$ are \emph{unbalanced}.  
A biased graph with all cycles balanced is said to be \emph{balanced}; otherwise it is \emph{unbalanced}.  
A biased graph with no balanced cycle is \emph{contrabalanced}.  
Observe that if $\GB$ is a balanced biased graph, then $F(G, \mathcal{B})$ is the cycle matroid $M(G)$ of $G$.  
We therefore view a graph as a biased graph with all cycles balanced.  
When no cycles are balanced $F(G, \emptyset)$ is the bicircular matroid of $G$ investigated by Matthews \cite{MR0505702}, Wagner \cite{MR815399}, and others (for instance, \cite{MR3100270, MR1892972}).  
The Dowling geometries \cite{MR0307951} arise precisely from those biased graphs for which the bias of cycles may be defined by associating an element of a finite group, and a direction, to each edge (see \cite{MR2017726}).

To begin our study of frame matroids of biased graphs with a balancing vertex, we describe the structure of the unbalanced cycles in such graphs.  
Let $G$ be a graph, let $P$ be a path in $G$, and let $Q$ be a path internally disjoint from $P$ linking two vertices $x, y \in V(P)$.  
We say the path $P'$ obtained from $P$ by replacing the subpath of $P$ linking $x$ and $y$ with $Q$ is obtained by \emph{rerouting} $P$ along $Q$.  

\begin{lem} \label{state:rerouting_paths}
Given two $u$-$v$ paths $P, P'$ in a graph, $P$ may be transformed into $P'$ by a sequence of reroutings.  
\end{lem}

\begin{proof}
Suppose $P$ and $P'$ agree on an initial segment from $u$, and let $x$ be the final vertex on this common initial subpath.  
If $x=v$, then $P=P'$, so assume $x \not= v$.  
Let $y$ be the vertex of $P'$ following $x$ that is also in $P$.  
Denote the subpath of $P'$ from $x$ to $y$ by $Q$.  
Since $y$ is different from $x$, the path obtained by rerouting $P$ along $Q$ has a strictly longer common initial segment with $P'$ than $P$.  
Continuing in this manner, eventually we find $x = v$; \ie, $P$ has been transformed into $P'$.  
\end{proof}

If subpath $R$ of path $P$ is rerouted along $Q$, and the cycle $R \cup Q$ is balanced, we refer to this as rerouting \emph{along a balanced cycle} or a \emph{balanced rerouting}.  
If $P$ is a path with distinct endpoints $x,y$ contained in a cycle $C$ and $Q$ is an $x$-$y$ path internally disjoint from $C$, and the cycle $P \cup Q$ is balanced, then the balanced rerouting of $P$ along $Q$ yields a new cycle $C'$.  
The following simple facts will be used extensively.  

\begin{lem} \label{lem:rerouting_along_a_bal_cycle}
Let $C$ be a cycle.  
If $C'$ is obtained from $C$ by rerouting along a balanced cycle, then $C$ and $C'$ have the same bias.  
\end{lem}

\begin{proof} 
Since $C \cup Q$ is a theta subgraph, this follows immediately from the theta property.  
\end{proof}

When the distinction is important, an edge that is not a loop is called a \emph{link}.  
The set of links incident to vertex $v$ is denoted $\delta(v)$.  

\begin{lem} \label{lem:ei_ej_with_balancing_vertex_have_same_parity_cycles}
Let $\GB$ be a biased graph and suppose $u$ is a balancing vertex in $\GB$.  
Let $\delta(u) = \{e_1, \ldots, e_k\}$.  
For each pair of edges $e_i, e_j$ $(1 \leq i < j \leq k)$, either all cycles containing $e_i$ and $e_j$ are balanced or all cycles containing $e_i$ and $e_j$ are unbalanced.  
\end{lem} 

\begin{proof} 
Fix $i, j$, and consider two cycles $C$ and $C'$ containing $e_i$ and $e_j$.  
Let $e_i = ux_i$ and $e_j = ux_j$.  
Write $C = u e_i x_i P x_j e_j u$ and $C' = u e_i x_i P' x_j e_j u$.  
Path $P$ may be transformed into $P'$ by a series of reroutings, $P \eq P_0, P_1, \ldots, P_l \eq P'$ in $G-u$.  
Since $u$ is balancing, each rerouting is along a balanced cycle.  
Hence by Lemma \ref{lem:rerouting_along_a_bal_cycle}, at each step $m \in \{1,\ldots, l\}$, the cycles $u e_i x_i P_{m-1} x_j e_j u$ and $u e_i x_i P_m x_j e_j u$ have the same bias.  
\end{proof} 

\begin{lem} \label{obs:exists_equiv_rel_on_delta_u}
Let $\GB$ be a biased graph and suppose $u$ is a balancing vertex in $\GB$.  
There exists an equivalence relation $\sim$ on $\delta(u)$ so that a cycle $C$ of $G$ containing $u$ is balanced if and only if it contains two edges from the same equivalence class.  
\end{lem} 

\begin{proof}
Define a relation $\sim$ on $\delta(u)$ by $e_i \sim e_j$ if there is a balanced cycle containing $e_i$ and $e_j$, or if $i = j$.  
Clearly $\sim$ is reflexive and symmetric; it is also transitive: 
Suppose $e_i \sim e_j$ and $e_j \sim e_t$; say $e_i = ux_i$, $e_j = ux_j$, and $e_t = u x_t$.  
Since there is a balanced cycle containing $x_i u x_j$ and a balanced cycle containing $x_j u x_t$, there is an $x_i$-$x_j$ path avoiding $u$ and an $x_j$-$x_t$ path avoiding $u$.  
Hence there is an $x_i$-$x_t$ path $P$ avoiding $u$ and a $P$-$x_j$ path $Q$ avoiding $u$.  
Let $P \cap Q = \{y\}$.  
Together, $u$, $e_i$, $e_j$, $e_t$, $P$, and $Q$ form a theta subgraph of $G$.  
By Lemma  \ref{lem:ei_ej_with_balancing_vertex_have_same_parity_cycles}, $u e_i x_i P y Q x_j e_j u$ and $u e_j x_j Q y P x_t e_t u$ are both balanced.  
By the theta property therefore, $u e_i x_i P x_t e_t u$ is balanced.  
Hence $e_i \sim e_t$.  
\end{proof}

We call the $\sim$ classes of $\delta(u)$ its \emph{unbalancing classes}.

\subsection{Signed graphs.} 
A convenient and well-studied way to assign biases to the cycles of a graph is by assigning a \emph{sign}, $+$ or $-$, to each of its edges.  
A cycle is then declared to be balanced \iiff it contains an even number of edges signed $-$.  
It is convenient to think of a signed graph as consisting of a graph $G$ together with a distinguished subset of edges $\Sigma \subseteq E(G)$ consisting of those edges signed $-$.  
We call $\Sigma$ the \emph{signature} of the graph.  
Thus a cycle $C$ is balanced \iiff $|E(C) \cap \Sigma|$ is even. 
Given a signature $\Sigma$ for a graph $G$, we write $\Bb_\Sigma$ for the collection of balanced cycles of $G$ given by $\Sigma$.  
We say that an arbitrary biased graph $(G,\Bb)$ \emph{is} a signed graph if there exists a set $\Sigma \subseteq E(G)$ so that $\Bb_{\Sigma} = \Bb$.  The following is a well-known characterisation of when this occurs.  

\begin{prop}[\cite{MR635702}] \label{prop:If_no_odd_theta} 
A biased graph is a signed graph \iiff it contains no contrabalanced theta subgraph.  
\end{prop}

\subsection{$k$-signed graphs.}  
Biased graphs with a balancing vertex have the biases of their cycles conveniently described using the following generalisation of signed graphs.  
Let $k$ be a positive integer.  
A \emph{$k$-signed graph} is a graph $G$ together with a collection of subsets of edges $\mathbf{\Sigma} = \{ \Sigma_1, \ldots, \Sigma_k \}$ with $\Sigma_i \cap \Sigma_j = \emptyset$ for $i \not= j$.  
Declare a cycle to be balanced \iiff $|E(C) \cap \Sigma_i|$ is even for every $1 \le i \le k$.  
We again call the collection $\mathbf{\Sigma}$ a \emph{signature} for $G$, and denote the collection of balanced cycles determined by $\mathbf{\Sigma}$ in this manner by $\Bb_{\mathbf{\Sigma}}$.  
We say that an arbitrary biased graph $(G,\Bb)$ \emph{is} a $k$-signed graph if there exists a collection $\mathbf{\Sigma} = \{\Sigma_1, \ldots, \Sigma_k\}$ so that $\Bb_{\mathbf{\Sigma}} = \Bb$.   

\begin{lem} \label{obs:relabelling_when_there_is_a_balancing_vertex}
Let $(G,\Bb)$ be a biased graph with a balancing vertex $u$ after deleting its set $U$ of unbalanced loops.  
Let $\{ \Sigma_1, \ldots, \Sigma_k \}$ be the partition of $\delta(u)$ into its unbalancing classes in $\GB \setminus U$, and let $\mathbf{\Sigma} = \{ U, \Sigma_1, \ldots, \Sigma_k \}$.  
Then $(G,\Bb)$ is a $k$-signed graph with $\Bb_{\mathbf{\Sigma}} = \Bb$ and $\GB$ is a $(k-1)$-signed graph with $\Bb_{\mathbf{\Sigma} \setminus \Sigma_i} = \Bb$ for every $1 \le i \le k$.  
\end{lem}

\begin{proof} 
This follows easily from the fact that $\sim$ is an equivalence relation in $\GB \setminus U$.  
\end{proof}

\section{Operations on biased graphs with a balancing vertex that preserve the frame matroid} 
\label{sec:operations} 
  
The operations of \emph{pinching} two vertices in a graph and of \emph{rolling up} an unbalancing class yield another biased graph representing the same frame matroid.  
We now describe these operations. 

\subsection{Pinching and splitting.} 
Let $H$ be a graph.  
Choose two distinct vertices $u, v \in V(H)$, and let $G$ be the graph obtained from $H$ by identifying $u$ and $v$ as a single vertex $w$.  
An edge with endpoints $u$ and $v$ becomes a loop incident to $w$, and so $\delta(w) = \delta(u) \cup \delta(v) \setminus \{e : e = uv\}$.  
Let $\Bb$ be the set of all cycles in $G$ not meeting both $\delta(u)$ and $\delta(v)$.  
It is easily verified (for instance, by checking all circuits of the two matroids) that $F\GB = M(H)$.  
We say the biased graph $\GB$ is obtained by \emph{pinching} $u$ and $v$, and call $\GB$ a \emph{pinch}. 
Biased graph $\GB$ is a signed graph: setting $\Sigma=\delta(u)$ or $\Sigma=\delta(v)$ gives a signature so that $\GB = (G, \Bb_\Sigma)$.  

The signed graph obtained by pinching two vertices of a graph to a single vertex $w$ has $w$ as a balancing vertex.  
Conversely, if $\GB$ is a signed graph with a balancing vertex $u$, then $\GB$ is obtained as a pinch of a graph $H$, which we may describe as follows.  
If $u$ is a cut vertex of $G$, then there are biased subgraphs $(G_1, \Bb_1)$, \ldots, $(G_m, \Bb_m)$ where each $(G_i, \Bb_i)$ has a balancing vertex $u_i$ ($i \in \{1, \ldots, m\}$), such that $u_i$ is not a cut vertex in $G_i$ and $\GB$ is obtained by identifying vertices $u_1$, \ldots, $u_m$ to a single vertex $u$.  
If any $(G_i,\Bb_i)$ has 
more than two unbalancing classes, 
then $\GB$ contains a contrabalanced theta, contradicting Proposition \ref{prop:If_no_odd_theta}.  
Hence for each $i$, 
$(G_i, \Bb_i)$ has at most two unbalancing classes, 
and Lemma \ref{obs:relabelling_when_there_is_a_balancing_vertex} gives a signature $\{ \Sigma_1^i, \Sigma_2^i \}$ (where the sets are permitted to be empty).  
Let $H$ be the graph obtained from $G$ by \emph{splitting} vertex $u$;  that is, 
\begin{itemize} 
\item replace $u$ with two vertices, $u_1$ and $u_2$, 
\item for $i \in \{1, \ldots, m\}$ redefine the incidence of each edge $e = v u \in \Sigma_1^i$ so that $e = v u_1$, 
\item redefine the incidence of each edge $e = v u \in \Sigma_2^i$ so that $e = v u_2$, 
\item set each unbalanced loop incident to $u$ as a link with endpoints $u_1$ and $u_2$, and 
\item place each balanced loop incident to $u$ as a balanced loop incident to either $u_1$ or $u_2$; 
\item leave all remaining edges' incidences unchanged. 
\end{itemize}
It is easily verified that $M(H)$ and $F\GB$ have the same set of circuits: 

\begin{prop} \label{prop:vertex_splitting_operation}
Let $(G,\Bb)$ be a signed graph with a balancing vertex $u$.  
If $H$ is obtained from $(G,\Bb)$ by splitting $u$, then $M(H) = F(G,\Bb)$.  
\end{prop}

It has thus far been convenient to denote a biased graph explicitly by the pair $\GB$ consisting of the underlying graph $G$ and its collection of balanced cycles $\Bb$. 
Sometimes this notation becomes cumbersome and it is more convenient to speak more concisely of a biased graph $\Om$; that is, to refer to the pair $\GB$ using the single symbol $\Om$. 
Moreover, it is often the case that the biases of cycles may be viewed as being given by a frame matroid that the biased graph represents (so the cycles that are balanced are those that are circuits of the matroid). 
In the case of $k$-signed graphs, it is more convenient to describe biases of cycles using a signature. 
In these instances there is no need to explicitly write $\GB$ for the pair of which the biased graph consists; indeed, doing so often makes the specification of the collection $\Bb$ redundant. 
Moreover, a biased graph may be thought of as a graph equipped with the extra information consisting of the biases of its cycles. Thus for a biased graph $\Om$, we may think of $\Om$ as having an underlying graph $G$ which is obtained by forgetting the biases of its cycles. 
We will therefore refer to vertices, edges, balanced cycles, induced subgraphs, induced biased subgraphs, and so on, of $\Om$, with the understanding that these belong to the underlying graph $G$ and the collection $\Bb$ of which $\Om$ consists. 
We reserve capital Greek letters $\Om$, $\Psi$ for biased graphs. 

\subsection{Roll-ups and unrolling.} 

If $\Om$ is a biased graph with a balancing vertex $u$, then the following \emph{roll-up} operation produces another biased graph with frame matroid isomorphic to $F(\Om)$.  
Let $\Sigma = \{e_1, \ldots, e_k\}$ be the set of edges of one of the unbalancing classes in $\delta(u)$.  
Let $\Om'$ be the biased graph obtained from $\Om$ by replacing each edge $e_i = u v_i  \in \Sigma$ with an unbalanced loop incident to its endpoint $v_i$.  
We say the biased graph $\Om'$ is obtained by a \emph{roll-up of unbalancing class $\Sigma$} of $\delta(u)$.  

Likewise, if a biased graph $\GB$ has a vertex $u$ that is balancing after deleting its set $U$ of unbalanced loops, and $\mathbf{\Sigma}$ is a signature for $G$ with $\Bb_{\mathbf{\Sigma}}=\Bb$ such that $\Sigma \setminus U \subseteq \delta(u)$, then the biased graph $(G',\Bb_{\mathbf{\Sigma}})$ obtained by replacing each unbalanced loop incident to $x \not= u$ with a $xu$ link  is obtained by \emph{unrolling} the set of unbalanced loops of $\GB$.  

Let $\Om_0$ be a biased graph with balancing vertex $u$ after deleting its set $U$ of unbalanced loops, and suppose that in $\Om_0 \setminus U$ there are $k$ unbalancing classes $\Sigma_1, \ldots, \Sigma_k$ in $\delta(u)$.  
Let $\Om$ be the biased graph obtained from $\Om_0$ by unrolling $U$, and write $\Sigma_0=U$.  
For each $i \in \{0, 1, \ldots, k\}$, let $\Om_i$ be the biased graph obtained from $\Om$ by rolling up unbalancing class $\Sigma_i$ (Figure \ref{fig:rollup}).  
Consider the set $\{\Om, \Om_0, \Om_1, \ldots, \Om_k\}$.  
We say any member of this set is a \emph{roll-up} of any other.  
\begin{figure}[tbp] 
\begin{center} 
\includegraphics[scale=0.9]{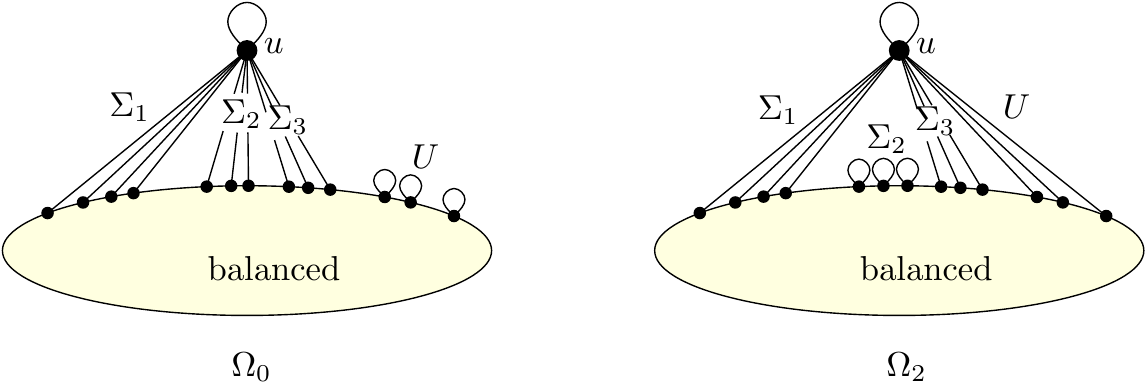}
\end{center} 
\caption{A roll-up; $F(\Om_0) = F(\Om_2)$}
\label{fig:rollup} 
\end{figure} 
It is straightforward to check that the frame matroids of any two roll-ups have the same set of circuits: 

\begin{prop} \label{prop:rollup_operation} 
Let $\Om$ be a biased graph with a balancing vertex after deleting its unbalanced loops.  
If $\Om'$ is a roll-up of $\Om$, then $F(\Om') = F(\Om)$.  
\end{prop}

Hence given a biased graph $\Om_0$ with a balancing vertex after deleting its set of unbalanced loops, the collection of biased graphs $\{\Om, \Om_0, \Om_1, \ldots, \Om_k\}$ is a set of $k+2$ representations of $F(\Om_0)$.  

Observe that if $H$ is a graph, then for each vertex $v \in V(H)$ the biased graph $\GB$ obtained by rolling up all edges in $\delta(v)$ has $F\GB = M(H)$.  
Conversely, if $\GB$ is balanced after deleting its set $U$ of unbalanced loops, then $U$ is a signature for $G$ such that $\Bb = \Bb_U$.  
Hence the graph $H$ obtained from $G$ by adding an isolated vertex $u$ and unrolling the edges in $U$ to $u$ has $M(H) = F\GB$.

\subsection{$H$-reductions and $H$-enlargements} 

The pinching/splitting and roll-up/\allowbreak unrolling operations have been known for some time. 
In this paper we introduce a new operation, that of \emph{$H$-reduction} and its inverse, \emph{$H$-enlargement}. 
We provide here motivation and an intuitive description; 
precise definitions are given in Section \ref{sec:representations}. 

A well-known operation that may be applied to a graph is that of decomposing along a 3-separation: if $G$ is the union of two subgraphs $G_1$ and $G_2$ and $V(G_1) \cap V(G_2) = \{x, y, z\}$, then we obtain two graphs $G_1'$ and $G_2'$ such that $G$ is a 3-sum of $G_1'$ and $G_2'$ as follows. For $i, j \in \{1,2\}$, $i \neq j$, we obtain $G_i'$ from $G$ by replacing subgraph $G_j$ with a triangle on vertices $x, y, z$. 
A triangle in $G$ is a circuit of size 3 in $M(G)$. 
In a biased graph $\Om$, there are, up to relabelling, 
four biased graphs representing a circuit of size 3: a balanced triangle, a pinch of a balanced triangle (a \emph{pinched triangle}), a roll-up of a balanced triangle (a \emph{rolled-up triangle}), and a contrabalanced theta consisting of three edges linking a pair of vertices (Figure \ref{fig:3circuits}). 
\begin{figure}[tbp] 
\begin{center} 
\includegraphics[scale=0.9]{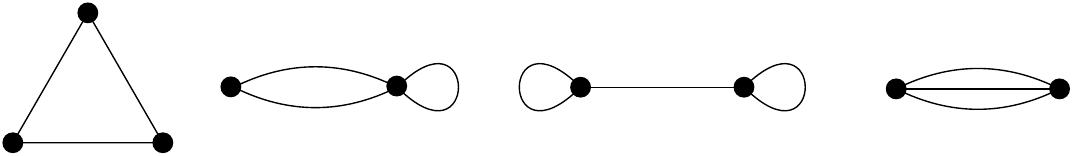}
\end{center} 
\caption{3-circuits in a biased graph}
\label{fig:3circuits} 
\end{figure} 

Whitney showed that a 3-connected graph uniquely represents its cycle matroid; his 2-isomorphism Theorem says that if $G$ has a subgraph $H$ that meets the rest of $G$ in exactly two vertices $x, y$, then redefining incidences so that all edges in $H$ incident to $x$ are instead incident to $y$ and those incident to $y$ are instead incident to $x$, while all other incidences remain unchanged, yields another graph representing $M(G)$. 

An $H$-reduction followed by an $H$-enlargement, together with an intermediate step, may be thought of as similar in spirit to these operations for graphs. 
Let $\Om$ be a biased graph, and let $H$ be a biased subgraph of $\Om$ such that either 
\begin{enumerate} 
\item $H$ is balanced and $V(H)$ meets the rest of $\Om$ in exactly three vertices $\{x,y,z\}$; 
\item $H$ is obtained a pinch of a graph $H'$ where vertices $x',x''$ of $H'$ are identified in the pinching operation to a single vertex $x$, and $H$ meets the rest of the biased graph in precisely 2 vertices $\{x,y\}$; or 
\item $H$ is a roll-up of a graph and meets the rest of the biased graph in precisely two vertices $\{x,y\}$. 
\end{enumerate} 
An $H$-reduction is one of the following operations. 
In case (1) replace $H$ with a balanced triangle on vertices $x, y, z$. 
In case (2) replace $H$ with a pinched triangle on vertices $x, y$, such that the unbalanced loop of the pinched triangle is incident to $x$. 
In case (3) replace $H$ with a rolled-up triangle on vertices $x, y$. 
An $H$-enlargement is the inverse operation of an $H$-reduction operation. 

It is convenient to combine a sequence of such reductions into a single reduction. 
If $H_1, \ldots, H_k$ are pairwise edge disjoint biased subgraphs of $\Om$ each satisfying the conditions for an $H_i$-reduction, then we write $H = \{H_1, \ldots, H_k\}$, perform an $H_i$ reduction for each $i \in \{1, \ldots, k\}$, and call the resulting biased graph an \emph{$H$-reduction}.  
We thus obtain a smaller biased graph $\Psi$. 
Now suppose there is another biased graph $\Psi'$ with $F(\Psi') = F(\Psi)$.  
We show that performing an $H_i$-enlargement on $\Psi'$ for each $i \in \{1, \ldots, k\}$ yields a biased graph $\Om'$ with $F(\Om') = F(\Om)$. 
We say $\Om'$ is an \emph{$H$-enlargement} of $\Psi'$. 

Theorem \ref{mainthm:main_bal_vertex_rep} says that given a 3-connected almost balanced biased graph $\Om$, we obtain every biased graph representation of $F(\Om)$ as either a roll-up of $\Om$ or as an $H$-enlargement in a sequence 
\[ \Om \mapsto \Psi \mapsto \Psi' \mapsto \Om' \] 
where $\Psi$ is an $H$-reduction of $\Om$ and has at most six vertices, $F(\Psi) = F(\Psi')$, and $\Om'$ is an $H$-enlargement of $\Psi'$. 
This is not only nice from a theoretical perspective. 
The fact that we are guaranteed to obtain a small reduced graph is good news from a practical standpoint, since it is not difficult to find all biased graphs $\Psi'$ whose frame matroid is isomorphic to $F(\Psi)$ when $|V(\Psi)| \leq 6$. 

The rest of this paper is devoted to a proof of Theorem \ref{mainthm:main_bal_vertex_rep}, and its two corollaries.

\section{Preliminaries} \label{sec:Preliminaries} 

A pair of parallel edges forming a balanced 2-cycle in $\Om$ is a pair of parallel elements in $F(\Om)$, as are two unbalanced loops incident to the same vertex.  
Since the biased subgraph consisting of the set of edges linking a pair of vertices $\{u,v\}$ has (each of $u$ and $v$ as) a balancing vertex, by Lemma \ref{obs:exists_equiv_rel_on_delta_u} the set of $u$-$v$ edges is partitioned into unbalancing classes, where each unbalancing class consists of a parallel class in $F(\Om)$. 
Let $\si(\Om)$ denote the simplification of $\Om$; that is, the biased graph obtained from $\Om$ by removing all balanced loops, all but one loop from the set of unbalanced loops incident to each vertex, and, for each pair of vertices, all but one edge of each parallel class in $F(\Om)$ of edges between them. 
If $\{\Om_1, \ldots, \Om_n\}$ is the set of all biased graphs representing $F(\si(\Om))$, then it is easy to find all biased graphs representing $F(\Om)$: to each biased graph $\Om_i$ each loop of the matroid may be added as a balanced loop incident to any vertex (or to a new vertex), and for each parallel class of $F(\Om)$, if in $\Om_i$ the representative $e$ of the class is an unbalanced loop incident to $v$ then all elements of the class are added as unbalanced loops incident to $v$; if $e$ is a $u$-$v$ link then all elements of the class are added as $u$-$v$ links forming balanced cycles with $e$. 
For this reason we may now assume that our biased graphs have no balanced loops, no balanced 2-cycles, and at most one unbalanced loop incident to each vertex.  

We use the following conventions in figures illustrating $k$-signed graphs.  
If $|\mathbf{\Sigma}|$ is at most three, then we use bold, dashed, or dotted edges to indicate edges in subsets $\Sigma_1$, $\Sigma_2$, $\Sigma_3$ of the signature, with edges in the same $\Sigma_i$ shown with same indication. 
A label near a vertex indicates that all edges incident to the vertex in the vicinity of the label are in the indicated subset of the signature.  
Most biased graphs we need to consider are $k$-signed.  
Otherwise, we resort to listing the balanced cycles of the graph.  
By assumption, all loops are unbalanced.  

Next we collect a few facts about what separations and hyperplanes look like in a biased graph representation of a frame matroid.

\subsection{Connectivity.}  \label{sec:intro_connectivity} 

We first summarise the standard notions of connectivity of graphs and matroids that we use, then consider connectivity of biased graphs.  
A \emph{separation} of a graph $G=(V,E)$ is a pair of edge disjoint subgraphs $G_1, G_2$ of $G$ with $G = G_1 \cup G_2$.  
The \emph{order} of a separation is $|V(G_1) \cap V(G_2)|$.  
A separation of order $k$ is a \emph{$k$-separation}.  
If both $V(G_1) \setminus V(G_2)$ and $V(G_2) \setminus V(G_1)$ are nonempty, then the separation is \emph{proper}.  
If $G$ has no proper separation of order less than $k$, then $G$ is \emph{$k$-connected}.  
The least integer $k$ for which $G$ has a proper $k$-separation is the \emph{connectivity} of $G$.  
(Note that highly connected graphs may contain loops or parallel edges.)  
A partition $(X,Y)$ of $E$ naturally induces a separation $G[X], G[Y]$ of $G$, which we also denote $(X,Y)$.  
We call $X$ and $Y$ the \emph{sides} of the separation.  
The \emph{connectivity function} of $G$ is the function $\lambda_G$ that to each partition $(X,Y)$ of $E$ assigns its order.  
That is, $\lambda_G(X,Y) = |V(X) \cap V(Y)|$.  

A \emph{separation} of a matroid $M$ on ground set $E$ is a partition of $E$ into two subsets $A$, $B$, and is denoted $(A,B)$; we call $A$ and $B$ the \emph{sides} of the separation.  
The \emph{order} of a separation $(A,B)$ is $r(A) + r(B) - r(E) + 1$.  
A separation of order $k$ with both $|A|, |B| \geq k$ is a \emph{$k$-separation}.  
If $M$ has no $l$-separation with $l<k$, then $M$ is \emph{$k$-connected}.   
The \emph{connectivity} of $M$ is the least integer $k$ such that $M$ has a $k$-separation; a matroid is \emph{connected} \iiff it has no $1$-separation.  
The \emph{connectivity function} of $M$ is the function $\lambda_M$ that assigns to each partition $(A,B)$ of $E$ its order; that is, $\lambda_M(A,B) = r(A) + r(B) - r(M) + 1$.  

If $(X,Y)$ is a partition of the edge set of a connected graph $G$, and each of the induced subgraphs $G[X]$ and $G[Y]$ is connected, then the orders of $(X,Y)$ in $G$ and in $M(G)$ are the same: $\lambda_{M(G)} = r(X) + r(Y) - r(M) + 1 = |V(X)|-1 + |V(Y)| - 1 - (|V|-1) + 1 = |V(X) \cap V(Y)| = \lambda_G(X,Y)$.  

A \emph{$k$-separation} of a biased graph $\Om=\GB$ is a $k$-separation of its underlying graph $G$, and the \emph{connectivity} of $\Om$ is that of $G$.  
The \emph{connectivity function} $\lambda_\Om$ of $\Om$ is that of $G$.  
As with graphs, a separation in a biased graph $\Om$ generally has a different order in its frame matroid $F(\Om)$ than in $\Om$.  
But let us assume $\Om$ is connected and unbalanced. Then  
\begin{equation}  \label{eqn:matroid_graph_conn}
\begin{aligned} 
\lambda_{F(\Om)}(X,Y) &= r(X) + r(Y) - r(F(\Om)) + 1 \\ &= |V(X)| - b(X) + |V(Y)| - b(Y) - |V| + 1 \\
&= |V(X) \cap V(Y)| - b(X) - b(Y) + 1 \\ &= \lambda_\Om(X,Y) - b(X) - b(Y) + 1 .
\end{aligned}  
\end{equation}
The following immediate consequences of equation (\ref{eqn:matroid_graph_conn}) will be useful. 
Assume that both sides of a separation $(X,Y)$ induce connected biased subgraphs.  
Then the difference between the order of $(X,Y)$ in $\Om$ and $F(\Om)$ is at most one.  
If $(X,Y)$ has order 1 in $\Om$, and one side induces a balanced biased subgraph, then $(X,Y)$ is a 1-separation of $F(\Om)$; if both sides are unbalanced, then $(X,Y)$ is a 2-separation of $F(\Om)$.  
If $(X,Y)$ has order 2 in $\Om$, then $(X,Y)$ is a 1-separation of the matroid if both sides are balanced, but a 3-separation if both sides are unbalanced.  
We will especially have occasion to use the following fact, which similarly follows immediately from equation (\ref{eqn:matroid_graph_conn}). 

\begin{lem} \label{lem:balancedsideof2sepa2sep} 
If $(X,Y)$ is a 2-separation of $\Om$ with $\Om[X]$ connected and balanced and $\Om[Y]$ connected and unbalanced, then $(X,Y)$ is a 2-separation of $F(\Om)$.  
\end{lem}

We now determine the forms a 2-separation of $F(\Om)$ may take in $\Om$ when $\Om$ is 3-connected. 

\begin{lem} \label{lem:2separationswithOmega3connected}
Let $\Om$ be a 3-connected unbalanced biased graph.  
If $(X,Y)$ is a 2-separation of $F(\Om)$, then either both $\Om[X]$ and $\Om[Y]$ are balanced and connected, or one side of the separation has size 2.  
Further, $\Om$ has the form of one of the biased graphs shown in Figure \ref{fig:2separationswithOmega3connected}. 
\end{lem}

\begin{figure}[tbp] 
\begin{center} 
\includegraphics[scale=0.9]{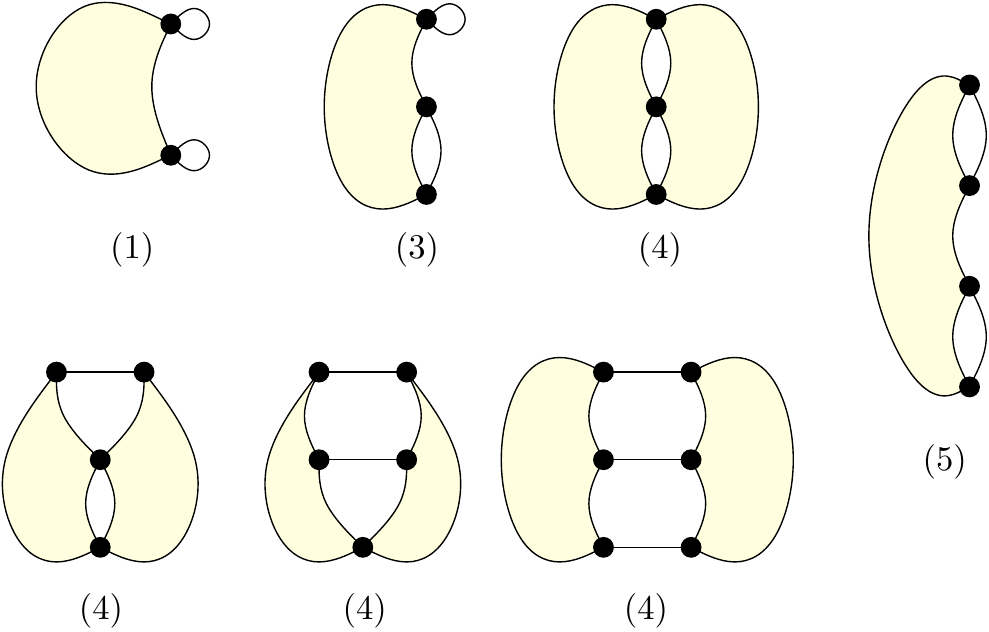}
\end{center} 
\caption{2-separations in $F(\Om)$ when $\Om$ is 3-connected. Each shaded bag of graph is balanced.}
\label{fig:2separationswithOmega3connected} 
\end{figure} 

\begin{proof} 
Let $S = V(X) \cap V(Y)$.  
Let $X_1, \ldots, X_m$ and $Y_1, \ldots, Y_n$ be the partitions of $X$ and $Y$, respectively, so that every induced biased subgraph $\Om[X_i]$, $\Om[Y_j]$ is a connected component of the biased subgraphs $\Om[X]$, $\Om[Y]$, respectively.  
Let us call these biased subgraphs \emph{parts}.  
Let $\delta_{X_i} = 1$ ($\delta_{Y_j}=1$) if $\Om[X_i]$ ($\Om[Y_j]$) is balanced and $\delta_{X_i}=0$ ($\delta_{Y_j}=0$) otherwise.  
Then $\lambda_{F(\Om)}(X,Y) = 2 = |S| + 1 - \sum_{i=1}^m \delta_{X_i} - \sum_{j=1}^n \delta_{Y_j}$.  
Since each vertex in $S$ is in exactly one part from each side, doubling each side of this equation and rearranging, we obtain 
\[ 2 = \sum_{i=1}^m \br{|S \cap V(X_i)| - 2\delta_{X_i}} + \sum_{j=1}^n \br{|S \cap V(Y_j)| - 2\delta_{Y_j}} .\] 
Since $\Om$ is 3-connected, parts that contain a vertex not in $S$ contain at least three vertices in $S$, and parts 
having all vertices in $S$ 
consist of either an unbalanced loop incident to a single vertex in $S$ or a single edge linking two vertices in $S$.  
Hence every term in the sums on the right-hand side of the equation above is nonnegative.  
Letting $t$ be the number of vertices in $S$ contained in a part, a balanced part contributes $t-2$ to the sum above, while an unbalanced part contributes $t$.  
Parts that are balanced with exactly two vertices in $S$ contribute 0 to the sum; let us call such a part \emph{neutral}. 
Since the total sum is two, the possibilities for the parts of $\Om[X]$ and $\Om[Y]$ (ignoring connectivity constraints for now) are: 
\begin{enumerate} 
\item  two unbalanced parts each with one vertex in S and all other parts neutral; 
\item  one unbalanced part with two vertices in $S$ and all other parts neutral; 
\item  one balanced part with three vertices in $S$, one unbalanced part with one vertex in $S$, and all other parts neutral; 
\item  two balanced parts with three vertices in $S$ and all other parts neutral; or
\item  one balanced part with four vertices in $S$ and all other parts neutral.  
\end{enumerate}
Since $\Om$ is 3-connected, the unbalanced parts in cases (1) and (3) each consist of a single loop, and case (2) cannot occur.  
The 3-connectedness of $\Om$ further implies that: in case (1), there is just one neutral part; in case (3) there is just one neutral part consisting of a single edge; in case (4) the neutral parts each consist of a single edge; and in case (5) there are exactly two neutral parts each consisting of a single edge.  
The possibilities are illustrated in Figure \ref{fig:2separationswithOmega3connected}, where each shaded bag of graph is balanced. 
More precisely, 
\begin{itemize} 
\item $V(X) \cap V(Y)$ contains only vertices illustrated in the figure as black discs, 
\item $V(X) \cap V(Y)$ has size 2 in case (1), size 3 in cases (3) and (4), and size 4 in case (5), and 
\item with the exception of loops all cycles contained entirely in one side of the separation are balanced. \qedhere
\end{itemize}
\end{proof}

The following is a straightforward corollary of Lemma \ref{lem:2separationswithOmega3connected}. 

\begin{lem} \label{lem:G_3conn_so_FG_3conn}
Let $\Om$ be a biased graph with no balanced loops, no balanced 2-cycles, and at most one unbalanced loop incident to a vertex. Further suppose that $\Om$ has a balancing vertex and $F(\Om)$ is nongraphic. If $\Om$ is 3-connected, then $F(\Om)$ is 3-connected. 
\end{lem} 

\begin{proof} 
In each of the possible cases (1), (3), (4), or (5) of Lemma \ref{lem:2separationswithOmega3connected}, either the biased graph cannot have a balancing vertex or its associated frame matroid is graphic.  
\end{proof}

\subsection{Cocircuits and hyperplanes in biased graphs.} 

The set of all edges incident to $v$ we denote by $\delta(v)^+$; that is, $\delta(v)^+ = \delta(v) \cup \{e : e$ is a loop incident to $v\}$.  
We denote by $\Omega - v$ the biased graph $(G - v, \Bb')$, where $\Bb'$ consists of all cycles in $\Bb$ that do not contain $v$.  

\begin{lem} \label{lem:v_not_bal_implies_star_at_v_a_cocircuit_1} 
Let $\Om$ be a 2-connected biased graph containing an unbalanced cycle.  
For each $v \in V(\Om)$, $\delta(v)^+$ is a cocircuit of $F(\Om)$ \iiff $v$ is not balancing.  
\end{lem}

\begin{proof} 
Let $n = |V(\Om)|$.  
Suppose $v \in V(\Om)$ is not a balancing vertex.  
Since the graph $\Om - v$ is connected and contains an unbalanced cycle, 
$r(E \setminus \delta(v)^+) = n-1 = r(F(\Om))-1$. 
If $e \in \delta(v)^+$, then $r(E \setminus \delta(v)^+ \cup \{e\}) = n$.  
Hence $E \setminus \delta(v)^+$ is a hyperplane, so $\delta(v)^+$ is a cocircuit.  
On the other hand, if $v \in V(\Om)$ is balancing then since $\Om-v$ is balanced and connected,  
$r(E \setminus \delta(v)^+) = (n-1) - 1 = n-2$.  
Thus $E \setminus \delta(v)^+$ is not a hyperplane.  
\end{proof}

\begin{lem} \label{lem:G3-conn_F(G-v)_disconn_implies_F(G-v)_graphic} 
Let $\Om$ be a 3-connected biased graph, and let $v \in V(\Om)$.  
If $F(\Om-v)$ is disconnected, then $\Om-v$ is a pinch and $E(\Om-v)$ is a graphic hyperplane of $F(\Om)$.  
\end{lem}

\begin{proof}
Choose a separation $(X,Y)$ of $F(\Om-v)$ with $\Om[X]$ and $\Om[Y]$ connected (Zaslavsky's characterisation of what the components of a frame matroid look like in a biased graph representation \cite[Theorem 2.8]{MR1088626} guarantees such a separation exists).  
A balanced cycle crossing the separation (\ie, containing an edge in each of $X$ and $Y$) would be a circuit of $F(\Om-v)$, so all such cycles are unbalanced.  
We claim each of $(\Om-v)[X]$ and $(\Om-v)[Y]$ are balanced.  
Suppose to the contrary that $(\Om-v)[X]$ is unbalanced; say $C \subseteq (\Om-v)[X]$ is an unbalanced cycle.  
Let $e \in Y$.  
Since $\Om-v$ is 2-connected, there are two disjoint paths linking $C$ and $e$. 
Together with $e$ and $C$, these paths form a theta subgraph $T$ of $G-v$.  
Then all three cycles in $T$ are unbalanced, so $T$ is a circuit of $F(\Om-v)$ containing an element of $X$ and an element of $Y$, a contradiction.  

We now show that $|V(X) \cap V(Y)| = 2$.  
Suppose for a contradiction that $|V(X) \cap V(Y)| > 2$.  
Let $x, y, z \in V(X) \cap V(Y)$.  
Since each of $\Om[X]$ and $\Om[Y]$ are connected, there is an $x$-$y$ path $P$ in $(\Om-v)[X]$, and an $x$-$y$ path $P'$ in $(\Om-v)[Y]$. 
Let $Q$ be a $P$-$z$ path in $(\Om-v)[X]$, and let $Q'$ be a $P'$-$z$ path in $(\Om-v)[Y]$ (where $Q$ or $Q'$ are allowed to be trivial).  
Then $P \cup P' \cup Q \cup Q'$ contains either a theta subgraph $T$ of $\Om-v$ in which all three cycles cross the separation, or a pair of handcuffs both cycles of which cross the separation.  
In either case, we have a circuit of $F(\Om-v)$ meeting both $X$ and $Y$, a contradiction.  

Together these facts imply that $\Om-v$ is a signed graphic pinch: a signature $\Sigma$ realising the biases of its cycle is obtained by choosing a vertex $x \in V(X) \cap V(Y)$ and setting $\Sigma$ to be all edges in $\delta(x) \cap X$.  
\begin{figure}[tbp]
\begin{center} 
\includegraphics[scale=0.9]{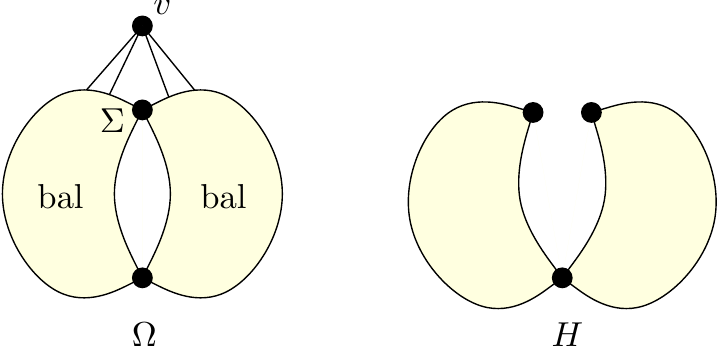}
\caption{$\Om$ is 3-connected, but $F(\Om-v)$ is disconnected.}  
\label{fig:G_3-conn_F(G-v)_disconnected}
\end{center}
\end{figure}
Splitting $x$, we obtain a graph $H$ with $M(H) = F(\Om-v)$ (Figure \ref{fig:G_3-conn_F(G-v)_disconnected}).  
Since $\Om-v$ is unbalanced, by Lemma \ref{lem:v_not_bal_implies_star_at_v_a_cocircuit_1} $\delta(v)^+$ is a cocircuit, so $E(\Om-v)$ is a hyperplane of $F(\Om)$.  
\end{proof}

\subsection{Committed vertices.} 
\label{sec:committedvertices} 

Let $M$ be a frame matroid represented by the biased graph $\Om = \GB$.  
For determining other possible biased graphs representing $M$ that are not isomorphic to $\Om$, the following observation is key.  

\begin{prop}[Slilaty, \cite{MR2245651}] \label{prop:comp_cocircuit_connnonbinhyperplane}
If $\Om$ is a connected biased graph with no balanced loops, then the complementary cocircuit of a connected nongraphic hyperplane of $F(\Om)$ consists precisely of the set of edges incident to a vertex.  
\end{prop}

Because Proposition \ref{prop:comp_cocircuit_connnonbinhyperplane} is central to our argument, we provide a proof for the convenience of the reader. 

\begin{proof}[Proof of Proposition \ref{prop:comp_cocircuit_connnonbinhyperplane}]
Call a set of edges whose removal results in a balanced biased graph a \emph{balancing} set.  
Since a cocircuit of $F(\Omega)$ is a minimal set of edges whose removal increases the number of balanced components by one, a cocircuit $D$ can be written as a disjoint union $D = S \cup B$ where $S = \emptyset$ or $S$ is a separating edge set of $\Omega$ and $B = \emptyset$ or $B$ is a minimal balancing set of an unbalanced component of $\Omega \setminus S$.  
If a biased graph has two components with nonempty edge sets, then its matroid cannot be connected (so a connected hyperplane in $F(\Omega)$ has at most one component in $\Omega$ with edges).  
Hence the complementary cocircuit of a connected hyperplane of $\Omega$ must be either the set of edges incident to a vertex or a minimal balancing set of $\Omega$.  
The frame matroid of a balanced biased graph is the cycle matroid of the graph, so if $X$ is a connected hyperplane whose complementary cocircuit 
is a minimal balancing set, then $X$ is graphic.  
Hence if $X$ is a connected nongraphic hyperplane of $F(\Om)$, then the complementary cocircuit of $X$ must be the set of edges incident to a vertex.  
\end{proof}

This motivates the following definition.  
\begin{defn} \label{defn:Committedvertex} 
A vertex $x \in V(\Om)$ is \emph{committed} if ${E \setminus \delta(x)^+}$ is a connected nongraphic hyperplane of $F(\Om)$.  
\end{defn}
If $\Om'$ is a biased graph with $F(\Om') = F(\Om)$, Proposition \ref{prop:comp_cocircuit_connnonbinhyperplane} says that for every committed vertex $x \in V(\Om)$, there is a vertex $x' \in V(\Om')$ with precisely the same set of incident edges.  
This notion is the main tool in our proof of Theorem \ref{mainthm:main_bal_vertex_rep}, so we will often wish to determine when a restriction of $F(\Om)$ to a set of edges $X$ is nongraphic.  
The easiest way to do this is to find an excluded minor for the class of graphic matroids in $\Om[X]$.  

Minors of biased graphs are defined as follows.  
Let $\Om=\GB$ be a biased graph, and $e \in E(G)$.  
The biased graph $\Om \backslash e$ is the biased graph $(G \backslash e, \Bb')$, where $\Bb' = \{ C : C \in \Bb$ and $e \notin C\}$.  
As long as $e$ is not a loop, the biased graph $\Om/e$ is the biased graph $(G/e,\Bb'')$, where $\Bb'' = \{C : C \cup e \in \Bb\}$.  
We have no need to delete or contract loops in this paper (definition are available in the literature, see for instance \cite[Sec.\ 6.10]{oxley:mt2011}).  
These operations are defined so that they are consistent with the corresponding minor operations in matroids: $F(\Om) \backslash e = F(\Om \backslash e)$ and $F(\Om)/e = F(\Om/e)$. 

The following lemma says that in a 3-connected biased graph, to determine that a vertex is committed it is enough to find a $U_{2,4}$ minor in the complement of the set of its incident edges.  

\begin{lem} \label{lem:x_committed_iff_U24} 
Let $\Om$ be 3-connected biased graph with a balancing vertex.  
Then $x \in V(\Om)$ is committed \iiff $F(\Om-x)$ is nonbinary.  
\end{lem}

\begin{proof} 
If $F(\Om-x)$ is graphic, then by definition $x$ is not committed.  
Conversely, suppose $x$ is not committed, \ie, $E \setminus \delta(x)^+$ fails to be a connected nongraphic hyperplane of $F(\Om)$.  
If $F(\Om-x)$ fails to be connected, then by Lemma \ref{lem:G3-conn_F(G-v)_disconn_implies_F(G-v)_graphic} it is graphic.  
If $F(\Om-x)$  
fails to be a hyperplane, then by Lemma \ref{lem:v_not_bal_implies_star_at_v_a_cocircuit_1} $x$ is balancing, so $F(\Om-x)$ is graphic.  
The remaining possibility is that $E \setminus \delta(x)$ is connected and a hyperplane, but graphic.  
\Ie, in any case, $F(\Om-x)$ is graphic.  
This shows $x$ is committed \iiff $F(\Om-x)$ is nongraphic.  

Now,  if $F(\Om-x)$ is nongraphic, then it can only be binary if it contains an $M^*(K_{3,3})$ or $M^*(K_5)$ minor (since neither $F_7$ nor $F_7^*$ are frame).  
But $\Om-x$ has the property that it is either balanced, in which case $F(\Om-x)$ is graphic and $x$ is uncommitted, or contains a balancing vertex.  
The property of having a balancing vertex is closed under deletion and contraction of links.  
It is not hard to see that $\Om-x$ has one of the biased graphs representing $M^*(K_{3,3})$ or $M^*(K_5)$ as a minor \iiff it has one of these biased graphs as a minor obtained by deleting or contracting only links.  
But none of the biased graph representations of $M^*(K_{3,3})$ and $M^*(K_5)$ has a balancing vertex (see Figure \ref{fig:biased_graphs_of_ex_min_for_graphic}), so $\Om-x$ cannot have either as a minor. 
Hence $F(\Om-x)$ is nongraphic \iiff $F(\Om-x)$ is nonbinary.  
\end{proof}
\begin{figure}[tbp] 
\begin{center} 
\includegraphics[scale=0.9]{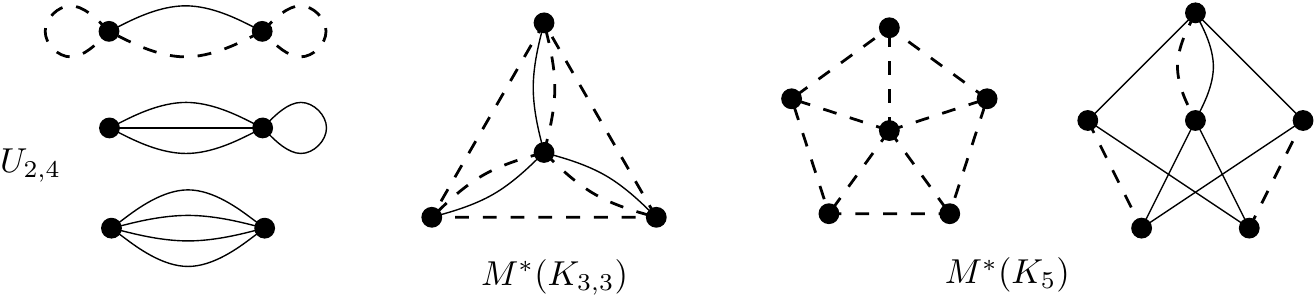}
\end{center} 
\caption{The biased graphs whose frame matroids are excluded minors for the class of graphic matroids.  The three biased graphs representing $U_{2,4}$ are contrabalanced.  The biased graphs with dashed edges are signed graphic, with signature indicated by dashed edges.}
\label{fig:biased_graphs_of_ex_min_for_graphic} 
\end{figure} 

By Lemma \ref{lem:x_committed_iff_U24}, when seeking to determine whether a vertex $x$ is committed, we just need find a $U_{2,4}$ minor in $F(\Om-x)$ or observe that none exists.  
The following lemma will help us find a $U_{2,4}$ minor.    
Let $x, y$ be a pair of vertices and $Q_1$, $Q_2$, and $Q_3$ be three internally disjoint paths $x$-$y$ paths comprising a contrabalanced theta graph $T$.  
We call $x$ and $y$ the \emph{branch} vertices of $T$.  
A \emph{shortcut} of $T$ is path $P$ linking any two of $\{Q_1, Q_2, Q_3\}$ and avoiding the third, such that neither endpoint of $P$ is a branch vertex (Figure \ref{fig:Odd_theta_with_shortcut}).  

\begin{lem} \label{lem:shortcut_an_odd_theta_U_24}
If a biased graph $\Om$ contains an countrabalanced theta with a shortcut, then $F(\Om)$ contains a $U_{2,4}$ minor.  
\end{lem}

\begin{proof} 
Consider the theta subgraph and shortcut $P$ shown in Figure \ref{fig:Odd_theta_with_shortcut}.  
\begin{figure}[tbp]
\begin{center}
\includegraphics[scale=0.9]{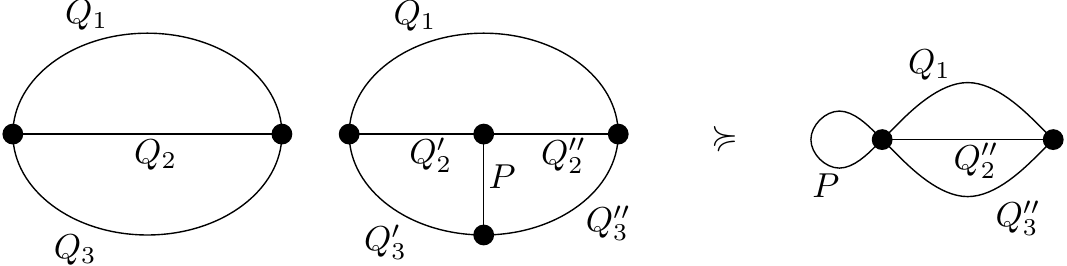}
\caption{An countrabalanced theta with a shortcut has a $U_{2,4}$ minor.}
\label{fig:Odd_theta_with_shortcut}
\end{center}
\end{figure}
By the theta property, one of $Q_2' P Q_3'$ or $Q_2'' P Q_3''$ is unbalanced, say \wolog{} 
$Q_2' P Q_3'$ unbalanced.  
Contracting $Q_2'$ and $Q_3'$ yields a biased graph representing $U_{2,4}$.  
\end{proof} 

We can immediately generalise Lemma \ref{lem:shortcut_an_odd_theta_U_24}.  

\begin{lem} \label{lem:odd_theta_and_unbal_cycle_U_24}
Suppose $M=F(\Om)$ and $\Om$ is connected.  
If $\Om$ contains a contrabalanced theta and an unbalanced cycle avoiding one of its branch vertices, then $M$ is nonbinary.  
\end{lem} 

\begin{proof} 
Let $Q_1, Q_2, Q_3$ be three internally disjoint $u$-$v$ paths forming a contrabalanced theta $T$, and let $C$ be an unbalanced cycle avoiding branch vertex $u$ of $T$.  
If there is a subpath of $C$ forming a shortcut of $T$, then by Lemma \ref{lem:shortcut_an_odd_theta_U_24}, $M$ has a $U_{2,4}$ minor.  
Otherwise, $C$ meets an internal vertex of at most one of $Q_1$, $Q_2$, or $Q_3$.  
Let $P$ be a $C$-$T$ path ($P$ is trivial if $C$ meets $T$).  
For $i \in \{1,2,3\}$, let $e_i \in Q_i$ be the edge in $T$ incident with $u$, and let $e_4$ be an edge in $C$ that is not in $T$.  
Contract all edges in $Q_1$, $Q_2$, and $Q_3$ except $e_1$, $e_2$, and $e_3$.  
Depending upon how $C$ meets $T$, we now have one of the biased graphs shown in Figure \ref{fig:odd_theta_and_unbal_cycle_U_24}.  
\begin{figure}[tbp] 
\begin{center} 
\includegraphics[scale=.9]{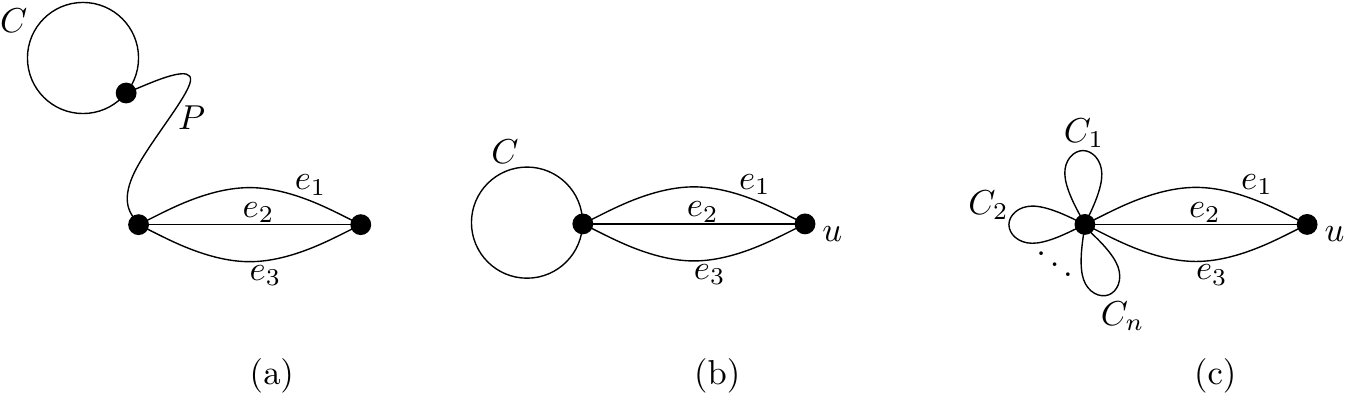}
\end{center} 
\caption[Finding $U_{2,4}$ (iii)]{If $\Om$ contains an odd theta and an unbalanced cycle avoiding one of its branch vertices.}
\label{fig:odd_theta_and_unbal_cycle_U_24} 
\end{figure} 
If $C$ is disjoint from $T$ (Figure \ref{fig:odd_theta_and_unbal_cycle_U_24}a) contract all edges of $P$ and all edges but $e_4$ remaining in $C$ to obtain a biased graph representing $U_{2,4}$.  
If $P$ is trivial, and after contracting all edges in $Q_1$, $Q_2$, and $Q_3$ excepting $e_1$, $e_2$, and $e_3$, the remaining edges in $C$ form a single cycle (Figure \ref{fig:odd_theta_and_unbal_cycle_U_24}b) then contract all edges but $e_4$ remaining in $C$ to obtain a biased graph representing $U_{2,4}$.  
If $P$ is trivial and after contracting the edges of $Q_1$, $Q_2$, and $Q_3$ excepting $e_1$, $e_2$, and $e_3$, the remaining edges of $C$ form more than one cycle (Figure \ref{fig:odd_theta_and_unbal_cycle_U_24}c) then contract all edges but $e_4$ in the cycle containing $e_4$ and delete all edges of $C$ left in the remaining cycles.  This again yields a biased graph representing $U_{2,4}$.  
\end{proof}

\subsection{$H$-reduction and $H$-enlargement.} \label{sec:representations} 

The following two lemmas are the keys to Theorem \ref{mainthm:main_bal_vertex_rep}.  

\begin{lem} \label{lem:possible_committed_lobe_reps}
Let $\Om$ be a biased graph with $F(\Om)$ 3-connected. 
Suppose $(X,Y)$ is a partition of $E(\Om)$ with $V(X) \cap V(Y) = \{u,v,w\}$, and suppose the biased subgraph $H$ of $\Om$ induced by $X$ is balanced, $V(H) \setminus \{u,v,w\} \not= \emptyset$, and that every vertex $x \in V(H) \setminus \{u,v,w\}$ is committed.  
Let $\Om'$ be a biased graph with $F(\Om') = F(\Om)$.  
Then the biased subgraph $H' \subseteq \Om'$ induced by $X$ is either 
\begin{enumerate} 
\item  balanced and isomorphic to $H$, 
\item  obtained from $H$ by pinching two vertices in $\{u,v,w\}$, or
\item  obtained from $H$ by rolling up all edges in $H$ incident to exactly one of $u$, $v$, or $w$.  
\end{enumerate}  
\end{lem} 

\begin{proof} 
Let the connected components of $H \setminus \{u,v,w\}$ be $H_1, \ldots, H_k$.  
Let $U_i, V_i, W_i$ be the set of neighbours of $u, v, w$, respectively, in $H_i$, for $i \in \{1, \ldots, k\}$.  
Since $F(\Om)$ is 3-connected and $H$ is balanced, 
each component $H_i$ contains a vertex in each of $U_i$, $V_i$ and $W_i$ (else $(E(H_i), E(\Om) \setminus E(H_i))$ would be a 1- or 2-separation of $F(\Om)$).  
Let $A = E(H) \cap \delta(u)$, and let $A_i$ be the set of edges in $A$ whose second endpoint is in $H_i$.  
We first show that for each $i \in \{1,\ldots, k\}$, every edge in $A_i$ is in $\Om'$ either incident to a common vertex or is an unbalanced loop.  
If $|A_i|=1$, the claim holds, so consider two edges $e, f$ in a set $A_i$.  
There is a path in $H_i$ linking the endpoints of $e$ and $f$ in $U_i$.  
This path together with $e,f,$ and $u$ is a balanced cycle $D$ in $\Om$, 
so $E(D)$ is a circuit in $F(\Om)$.  
Since every vertex in $D-u$ is committed, this implies that in $\Om'$ either both $e$ and $f$ are incident to a common vertex or are both unbalanced loops.  
Similarly, define $B$ to be the set of edges in $E(H) \cap \delta(v)$ and $C = E(H) \cap \delta(w)$, and define $B_i$ (resp.\ $C_i$) to be the set of edges in $B$ (resp.\ $C$) whose second endpoint is in $H_i$.  
The analogous argument shows that in $\Om'$, for each $i \in \{1, \ldots, k\}$, either all edges in $B_i$ (resp.\ $C_i$) are incident to a common vertex or are all unbalanced loops.  

Now for each $i \in \{1, \ldots, k\}$, let $H_i'$ be the biased subgraph of $\Om'$ induced by the elements of $F(\Om)$ in $H_i \cup A_i \cup B_i \cup C_i$.  
Since every vertex $x \in V(H) \setminus \{u,v,w\}$ is committed, for each   vertex $x \in V(H_i)$ there is a unique vertex $x' \in V(H_i')$ with $\delta(x')=\delta(x)$.  
Let $U_i', V_i', W_i'$ be the sets of vertices $x'$ of $H'$ whose corresponding vertices $x$ of $H$ are in $U_i, V_i, W_i$, respectively.  
Suppose first that none of $A_i, B_i$, or $C_i$ consist of unbalanced loops in $\Om'$: each edge in $A_i$ has an endpoint in $U_i'$ and a common second endpoint $u'$, each edge $B_i$ has an endpoint in $V_i'$ and a common second endpoint $v'$, and each edge in $C_i$ has an endpoint in $W_i'$ and a common second endpoint $w'$.  
Now it may be that in $\Om'$ all three of $u', v', w'$ are distinct, or that some two of $v', u', w'$ are the same vertex.  
It cannot be that $u'=v'=w'$:  if so, let $P$ be a $u$-$v$ path and $Q$ be a $P$-$w$ path in $H$; then $E(P \cup Q)$ is independent in $F(\Om)$ but would be dependent in $F(\Om')$, a contradiction.  

We now claim that at most one of $A_i$, $B_i$, or $C_i$ are unbalanced loops in $\Om'$.  
For suppose to the contrary that the edges representing the elements in both $A_i$ and $B_i$ are unbalanced loops in $\Om'$.  
There is a $u$-$v$ path $P$ in $H_i$; $E(P)$ is independent in $F(\Om)$, but a circuit in $F(\Om')$, a contradiction.  
Similarly, not both $A_i$ and $C_i$, nor both $B_i$ and $C_i$, may be unbalanced loops.  

Now suppose that in $\Om'$ the edges in $A_i$ are unbalanced loops, the edges in $B_i$ are incident to a common vertex $v'$, and the edges in $C_i$ are incident to a common vertex $w'$.  
We claim that $v' \not= w'$.  
For supposing $v' = w'$, then, as in the previous paragraph, choosing a $u$-$v$ path $P$ and a $P$-$w$ path $Q$ in $H_i$ yields a set $E(P \cup Q)$ independent in $F(\Om)$ but dependent in $F(\Om')$.  
Similarly, if a set $B_i$ (resp.\ $C_i$) consists of unbalanced loops in $\Om'$, then the common endpoint $u'$ of the edges in $A_i$ and the common endpoint $w'$ of the edges in $C_i$ (resp.\ $v'$ of edges in $B_i$) are distinct in $\Om'$.  

Hence each biased subgraph $H_i'$ has the form of one of the biased graphs (a)-(g) shown in Figure \ref{fig:The_H_is}.  
\begin{figure}[htbp] 
\begin{center} 
\includegraphics[scale=0.9]{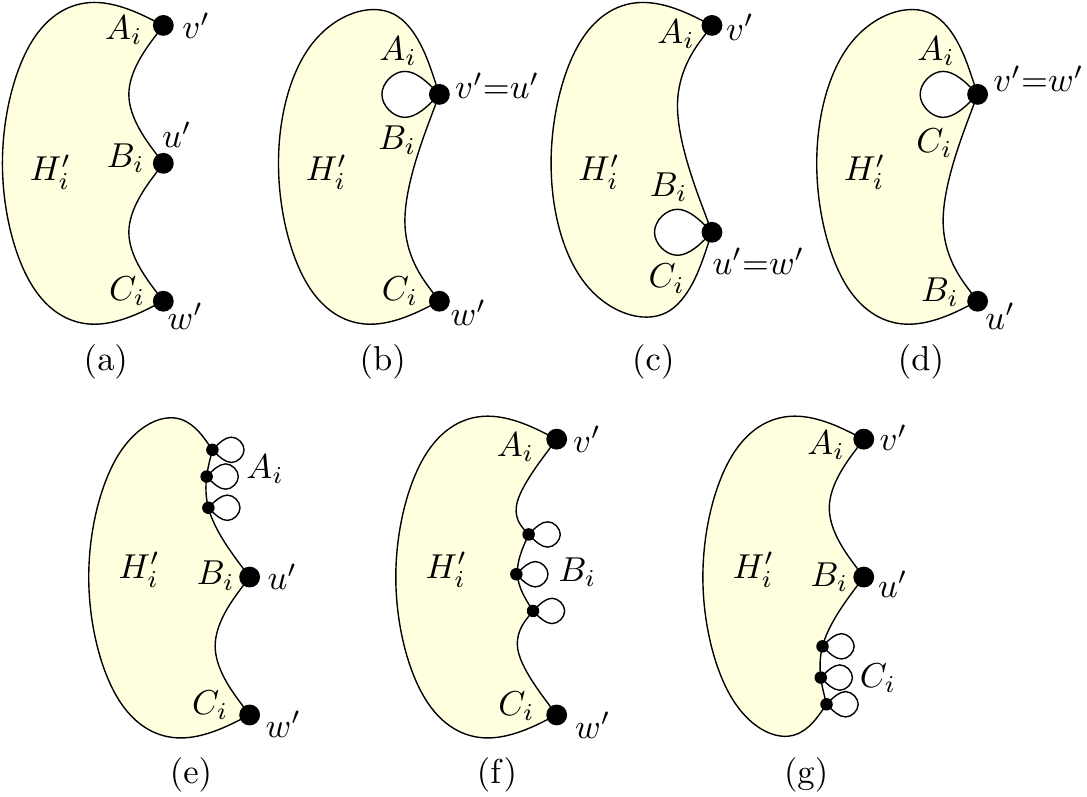}
\end{center} 
\caption[Possible representations of large balanced or pinched subgraphs]{Possible biased graph representations of $F(H)$ when $H$ is a balanced or pinched biased subgraph all of whose vertices aside from $v', u', w'$ are committed.}
\label{fig:The_H_is} 
\end{figure} 
It is now easy to see that if for some $i \not= j$, $H_i'$ and $H_j'$ are not both of the same form (a)-(g), then $F(\Om') \not= F(\Om)$: we would have a balanced cycle in $H$ the elements of which would form an independent set in $F(\Om')$.  
Hence $\bigcup_i H_i'$ itself has the form of one of these biased graphs.  
It is also now easy to see that edges in $H$ that link pairs of vertices in $\{u,v,w\}$ must be placed in $\Om'$ in the same form as the $H_i'$. 
For instance, if the $H_i'$ are of the form shown in Figure \ref{fig:The_H_is}(b), then a balanced triangle on $u,v,w$ in $H$ must be a pinched triangle on the vertices $\{v' \eq u', w'\}$ with the $vu$ edge in $\Om$ now an unbalanced loop incident to $v' \eq u'$.  
The conclusion now follows:  
If $H'$ is of the form shown in Figure \ref{fig:The_H_is}(a), then $H'$ is balanced and isomorphic to $H$.  
If $H'$ is of the form (b)-(d), then $H'$ is obtained from $H$ by pinching two of $\{u,v,w\}$, and if $H'$ is one of (e)-(g), then $H'$ is obtained from $H$ as a roll-up of the edges of $H$ incident to one of $u$, $v$, or $w$.  
\end{proof} 

\begin{lem} \label{lem:possible_committed_lobe_reps_pinched}
Let $\Om$ be a biased graph with $F(\Om)$ 3-connected. 
Suppose $(X,Y)$ is a partition of $E(\Om)$ with $V(X) \cap V(Y) = \{u,v\}$, and suppose that the biased subgraph $H$ of $\Om$ induced by $X$ is a pinch with signature $\{\Sigma_1, \Sigma_2\} \subseteq \delta(u)$, that $V(H) \setminus \{u,v\} \not= \emptyset$, and that every vertex $x \in V(H) \setminus \{u,v\}$ is committed.  
Let $H''$ be the graph obtained by splitting $u$, with $\delta(u_1) \cup \delta(u_2) = \delta(u)$.  
Let $\Om'$ be a biased graph with $F(\Om') = F(\Om)$.  
Then the biased subgraph $H' \subseteq \Om'$ induced by $X$ is either 
\begin{enumerate} 
\item  balanced and isomorphic to $H''$, 
\item  obtained from $H''$ by pinching two vertices in $\{u_1, u_2, v\}$, or
\item  obtained from $H''$ by rolling up all edges in $H''$ incident to exactly one of $u_1$, $u_2$, or $v$.  
\end{enumerate}  
\end{lem} 

\begin{proof} 
By Proposition \ref{prop:vertex_splitting_operation}, $F(H'') = F(H)$.  
The proof is that of Lemma \ref{lem:possible_committed_lobe_reps}, with $H''$ taking the place of $H$ and $u_1, u_2, v$ taking the place of $u, v, w$, respectively.  
\end{proof}

Let $\Om$ be a biased graph with a balancing vertex $u$, with $F(\Om)$ 3-connected.  
Let $(X,Y)$ be a partition of $E(\Om)$, let $S = V(X) \cap V(Y)$, and let $H$ be the biased subgraph of $\Om$ induced by $X$.   
Suppose that $V(H) \setminus S \not= \emptyset$, that every vertex $x \in V(H) \setminus S$ is committed, and that one of the following holds: 
\begin{enumerate} 
\item  $S = \{u, v, w\}$ for some $v, w \in V(\Om)$, and $H$ is balanced, or 
\item  $S = \{u,v\}$ for some $v \in V(\Om)$ and $H$ is a pinch with signature $\Sigma \subseteq \delta(u)$.  
\end{enumerate} 
An \emph{$H$-reduction} is one of the following operations.  
In case (1), replace $H$ in $\Om$ with a balanced triangle on $\{u,v,w\}$.  
In case (2), replace $H$ in $\Om$ with pinched triangle consisting of two $u$-$v$ edges and an unbalanced loop on $u$.  
Likewise, if $H_1, \ldots, H_k$ are pairwise edge disjoint biased subgraphs of $\Om$ each satisfying the conditions for an $H_i$-reduction, then we write $H = \{H_1, \ldots, H_k\}$, perform an $H_i$ reduction for each $i \in \{1, \ldots, k\}$, and call the resulting biased graph an \emph{$H$-reduction}.  
We call each such balanced or pinched subgraph $H_i$ a \emph{lobe} of $\Om$.  
An $H$-reduction of $\Om$ is denoted $\re(\Om,H)$.  

Suppose $\re(\Om,H)$ is obtained via replacement of lobes $H_1, \ldots, H_k$, and $\Psi$ is a biased graph with $F(\Psi) = F(\re(\Om,H))$. 
A biased graph $\Om'$ with $F(\Om') = F(\Om)$ may be obtained from $\Psi$ as follows.  
For $i \in \{1, \ldots, k\}$, let $C_i$ be the 3-circuit of $F(\re(\Om,H))$ that replaced lobe $H_i$ in $\Om$.  
If $C_i$ is a balanced triangle, a pinched triangle, or a rolled-up triangle in $\Psi$, then replace $C_i$ in $\Psi$ with a biased subgraph $H_i'$ of one of the three forms given by Lemma \ref{lem:possible_committed_lobe_reps} or \ref{lem:possible_committed_lobe_reps_pinched}: 
\begin{enumerate} 
\item  If $C_i$ is a balanced triangle in $\Psi$, replace $C_i$ by a balanced biased subgraph $H_i'$, where $H_i'$ is a copy of the balanced subgraph $H_i$ or, in the case $H_i$ is a pinch, a copy of the graph $H_i''$ obtained by splitting $u$.  
\item  If $C_i$ a pinched triangle, replace $C_i$ with a biased graph $H_i'$ obtained from $H_i$ or $H_i''$ by pinching two of its vertices in $\{u,v,w\}$ or $\{u_1, u_2, v\}$, respectively.  
\item  If $C_i$ is a rolled-up triangle, replace $C_i$ with a biased graph $H_i'$ obtained from $H_i$ or $H_i''$ by a roll-up of edges incident to a vertex in $\{u,v,w\}$ or $\{u_1, u_2, v\}$, respectively.  
\end{enumerate} 
In each case, the replacement is done by deleting $E(C_i)$ from $\Psi$ and identifying each vertex of $\Psi$ previously incident to an edge in $C_i$ with a vertex of $H_i'$ appropriately.  
Which pairs of vertices to identify are chosen as follows.  
Suppose 3-circuit $abc$ in $\Psi$ is to be replaced by a biased graph $H'$ of one of the forms given by Lemma \ref{lem:possible_committed_lobe_reps} or \ref{lem:possible_committed_lobe_reps_pinched}.  
As in the proofs of Lemma \ref{lem:possible_committed_lobe_reps} and \ref{lem:possible_committed_lobe_reps_pinched}, let $A=\delta(u) \cap E(H)$, $B=\delta(v) \cap E(H)$, and $C=\delta(w) \cap E(H)$ if $H$ is balanced in $\Om$, or if $H$ is a pinch in $\Om$, let $A=\delta(u_1) \cap E(H'')$, $B=\delta(u_2) \cap E(H'')$, and $C=\delta(v) \cap E(H'')$, where $H''$ is obtained by splitting vertex $u$ and $u_1$, $u_2$ are the resulting new vertices of $H''$.  
Let 
\[ v_A = \begin{cases} u &\text{if } H \text{ is balanced} \\ u_1 &\text{if } H \text{ is a pinch,} \end{cases} \ \ \ 
v_B = \begin{cases} v &\text{if } H \text{ is balanced} \\ u_2 &\text{if } H \text{ is a pinch,} \end{cases} \ \ \ \] and \[
v_C = \begin{cases} w &\text{if } H \text{ is balanced} \\ v &\text{if } H \text{ is a pinch.} \end{cases} \]
Each edge in the 3-circuit $abc$ in $F(\re(\Om,H))$ corresponds to a path in $\Om$ linking pairs of vertices in $\{v_A, v_B, v_C\}$, with $a$ corresponding to a $v_A$-$v_B$ path, $b$ a $v_B$-$v_C$ path, and $c$ a $v_C$-$v_A$ path.  
Indeed, circuit $abc$ in $\re(\Om,H)$ may be obtained as a minor of $\Om$ from such paths.  
If in $\Psi$, edges $a$ and $b$ share a common endpoint $x_{ab}$, edges $b$ and $c$ share common endpoint $x_{bc}$, and edges $a$ and $c$ share endpoint $x_{ac}$, then construct $\Om'$ by identifying vertex $v_B$ with $x_{ab}$, vertex $v_C$ with $x_{bc}$, and vertex $v_A$ with $x_{ac}$.  
Observe that in the case $abc$ is a pinched triangle, two of $x_{ab}$, $x_{bc}$, $x_{ac}$ are the same vertex, thus $H'$ is a pinch in $\Om'$, as desired.  
If $abc$ is a rolled-up triangle, and so has two edges, say $a$ and $c$, that do not share an endpoint, then again identify vertex $v_B$ with $x_{ab}$ and vertex $v_C$ with $x_{bc}$, and roll-up the edges in $A$.  
We call the biased graph $\Om'$ resulting from carrying out this procedure for each 3-circuit $C_i$ ($i \in \{1, \ldots, k\}$) that is not a contrabalanced theta an \emph{$H$-enlargement} of $\Psi$.  

Figures \ref{fig:G-u_bal_G-v_3simclasses_reps_intro} and \ref{fig:G-u_bal_G-v_3simclasses_reps_with_bags_of_graph_intro} provide an example of this process.  
\begin{figure}[tbp] 
\begin{center} 
\includegraphics[scale=0.9]{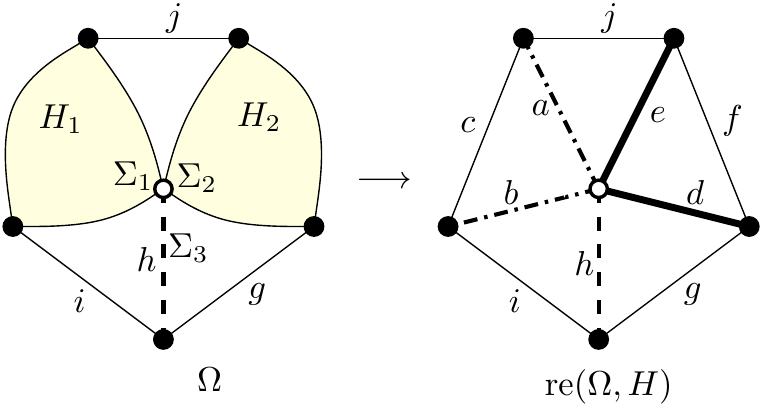}
\end{center} 
\caption{
$\Om$ and $\re(\Om,H)$. 
}
\label{fig:G-u_bal_G-v_3simclasses_reps_intro} 
\end{figure} 
\begin{figure}[tbp] 
\begin{center} 
\includegraphics[scale=0.9]{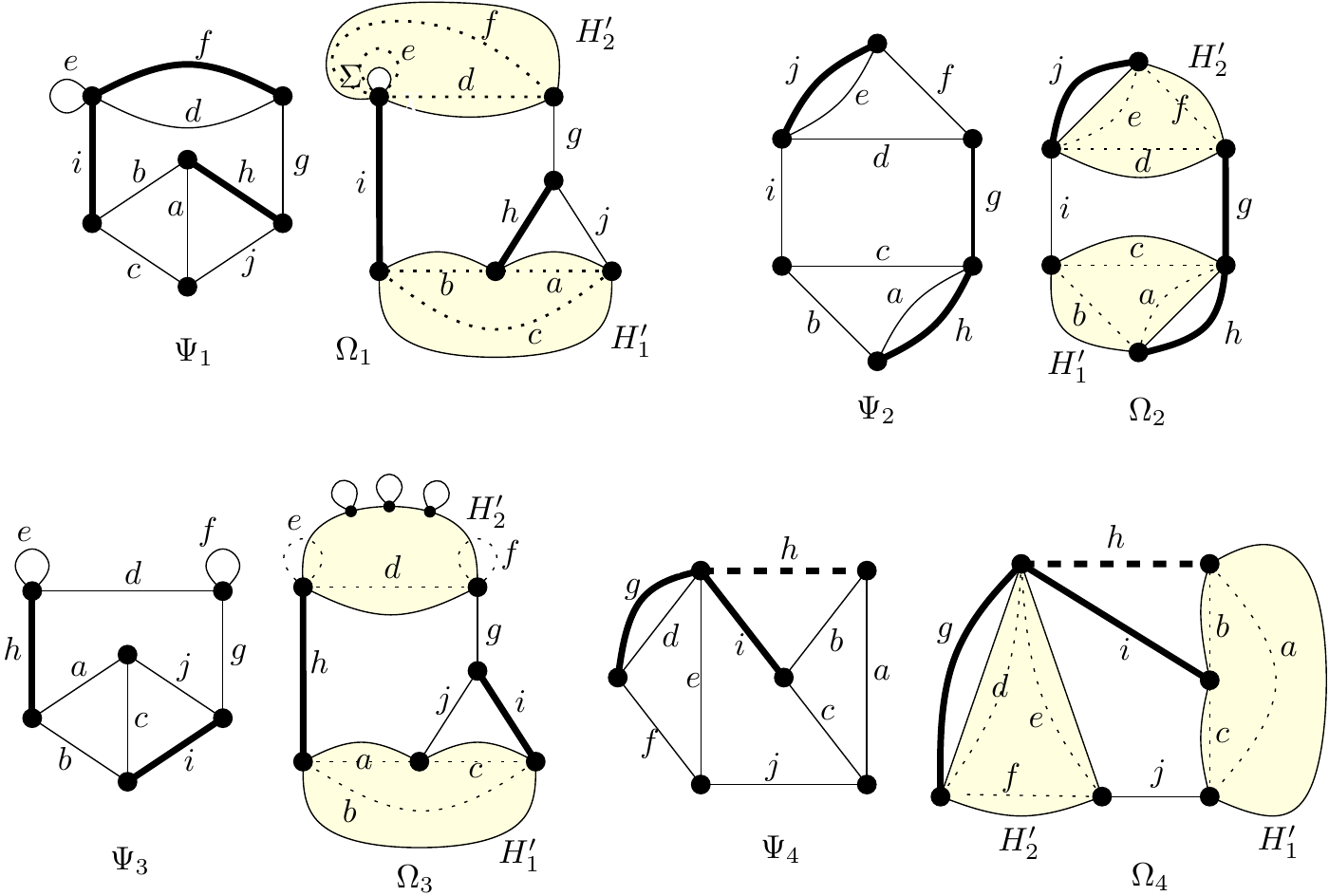}
\end{center} 
\caption{Biased graphs representing $F(\Om)$ obtained as $H$-enlargements.}
\label{fig:G-u_bal_G-v_3simclasses_reps_with_bags_of_graph_intro} 
\end{figure} 
Figure \ref{fig:G-u_bal_G-v_3simclasses_reps_intro} shows the $H$-reduction of a biased graph $\Om$ (whose balancing vertex is indicated as a white disc). 
Figure \ref{fig:G-u_bal_G-v_3simclasses_reps_with_bags_of_graph_intro} shows $H$-enlargements $\Om_1$, $\Om_2$, $\Om_3$, $\Om_4$ of four biased graphs $\Psi_1$, $\Psi_2$, $\Psi_3$, $\Psi_4$, respectively, each with $F(\Psi_i) \cong F(\re(\Om,H))$. 
Each of these $H$-enlargements $\Om_1$, $\Om_2$, $\Om_3$, $\Om_4$ has $F(\Om_i) \cong F(\Om)$. 

The following lemma will enable us to show that aside from roll-ups, all biased graphs representing $F(\Om)$ are obtained as $H$-enlargements.  

\begin{lem} \label{lem:H_enlargements_are_all}
Let $\Om$ be a biased graph with $F(\Om)$ 3-connected.  
Suppose for $i \in \{1, \ldots, k\}$, $(X_i, Y_i)$ is a 
partition of $E(\Om)$ and $\bigcap_i X_i = \emptyset$.  
Let $H_i = \Om[X_i]$ be the biased subgraph induced by $X_i$, and let $S_i = V(X_i) \cap V(Y_i)$.  
Suppose for each $i \in \{1, \ldots, k\}$, either $|S_i|=3$ and $H_i$ is balanced, or $|S_i|=2$ and $H_i$ is a pinch with its balancing vertex contained in $S_i$.  
Suppose further that $V(H_i) \setminus S_i$ is nonempty, that every vertex $x \in V(H_i) \setminus S_i$ is committed, and that there is no additional partition $(X,Y)$ satisfying these conditions.  
Let $H=\{H_1, \ldots, H_k\}$. 
If $\Om'$ is a biased graph representing $F(\Om)$, then $\Om'$ is an $H$-enlargement of a biased graph $\Psi$ with $F(\Psi) = F(\re(\Om,H))$.  
\end{lem} 

\begin{proof} 
Biased graph $\re(\Om,H)$ is a minor of $\Om$, say $\Om \setminus S / T = \re(\Om,H)$.  
Then $F(\Om) \setminus S / T = F(\re(\Om,H))$, and $F(\Om') \setminus S / T = F(\Psi)$, where $\Psi = \Om' \setminus S / T$.  
Since $F(\Om') = F(\Om)$, we have $F(\re(\Om,H)) = F(\Psi)$, so this produces a biased graph $\Psi$ with $F(\Psi) = F(\re(\Om,H))$.  
For each $i \in \{1, \ldots, k\}$, in $F(\re(\Om,H))$ there is a circuit $C_i$ of size three resulting from the minor operations which brought $\Om$ to $\re(\Om,H)$; \ie, for some $H_i \subseteq \Om$, $C_i = H_i \setminus S / T$ in $\re(\Om,H)$.  
By Lemma \ref{lem:possible_committed_lobe_reps} or \ref{lem:possible_committed_lobe_reps_pinched}, the set of edges $X_i$ in $\Om'$ induces a biased subgraph $H_i'$ of one of types 1, 2, or 3, as described in Lemma \ref{lem:possible_committed_lobe_reps} or \ref{lem:possible_committed_lobe_reps_pinched}.  
In $F(\Psi)$, $C_i$ forms a circuit of size 3.  
Replacing $C_i$ with
(1) a balanced subgraph isomorphic to $H_i$ or $H_i''$ if $C_i$ is a balanced triangle, 
(2) a pinch of two vertices in $\{u,v,w\}$ of $H_i$ or $\{u_1,u_2,v\}$ of $H_i''$ if $C_i$ a pinched triangle, or 
(3) a roll-up of $H_i$ from one of $u$, $v$, or $w$ or of $H_i''$ from one of $u_1$, $u_2$, or $v$, if $C_i$ is a rolled-up triangle, yields $\Om'$. 
\end{proof} 

Hence given the maximal collection of edge disjoint lobes $H = \{H_1, \ldots, H_k\}$ of $\Om$, we obtain all biased graphs representing $F(\Om)$ as $H$-enlargements of biased graphs with frame matroids isomorphic to $F(\re(\Om,H))$.  
To find all biased graphs representing $F(\Om)$ therefore, we just need find all biased graphs with frame matroids isomorphic to $F(\re(\Om,H))$.  

Theorem \ref{mainthm:main_bal_vertex_rep} asserts that in all cases, $\re(\Om,H)$ is small enough that this is not difficult.  
Since $\re(\Om,H)$ is small, all biased graphs $\Psi$ with $F(\Psi) = F(\re(\Om,H))$ may be easily determined by an exhaustive search.  
For instance, this can be done as follows.  
Set $n=\rank(F(\re(\Om,H)))$, and list the circuits $\Cc = \{C_1, C_2, \dots, C_k\}$ of $F(\re(\Om,H))$ of size at most $n$.  
List all biased graph representations on at most $n$ vertices of $C_1$.  
For each of these, list all biased graphs on at most $n$ vertices whose edges represent both $C_1$ and $C_2$ but contain no circuit not in $\Cc$.  
For each of the resulting biased graphs, list all biased graphs on $n$ vertices whose edges represent $C_1$, $C_2$, and $C_3$, but contain no circuit not in $\Cc$.  
Continuing in this manner, after $k$ steps we obtain a list of all biased graphs whose circuits are precisely those in $\Cc$.  
Since $n \leq 6$, this process is practical even by hand (if somewhat labourious).  

\subsection{Proofs of Corollaries \ref{maincor:notsomanybiasedgraphreps} and \ref{maincor:fourconnbalvertexrepunique}} 
\label{sec:proofsofcorollaries} 

In the course of proving Theorem \ref{mainthm:main_bal_vertex_rep}, it is shown that there are only a finite number of $H$-reductions possible.  
Counting the number of biased graphs representing the matroid in each case yields Corollary \ref{maincor:notsomanybiasedgraphreps}.   
We do not do this counting here.  
The required biased graphs are exhibited in Chapter 5 of \cite{funkthesis}.  
Alternatively, the interested reader may verify Corollary \ref{maincor:notsomanybiasedgraphreps} by producing the required biased graphs using a method such as that described in the previous paragraph.  

Since the application of an $H$-reduction depends upon the existence of a biased subgraph $H_i \in H$ of $\Om$ for which $\br{E(H_i), E(\Om) \setminus E(H_i)}$ is a 3-separation, Corollary \ref{maincor:fourconnbalvertexrepunique} also follows immediately from Theorem \ref{mainthm:main_bal_vertex_rep}:  

\begin{proof}[Proof of Corollary \ref{maincor:fourconnbalvertexrepunique}] 
By Theorem \ref{mainthm:main_bal_vertex_rep}, if $\Om'$ is a biased graph with $F(\Om') = F(\Om)$ that is not obtained as a roll-up of $\Om$, then $F(\Om)$ has a 3-separation.  
\end{proof}

\section{Proof of Theorem \ref{mainthm:main_bal_vertex_rep}}

We are now ready to prove Theorem \ref{mainthm:main_bal_vertex_rep}.  
The proof proceeds through several cases.  
Let $\Om$ be a 3-connected biased graph with a balancing vertex $u$, with $F(\Om)$ nongraphic and 3-connected. 
(By Lemma \ref{lem:G_3conn_so_FG_3conn}, the assumption that $F(\Om)$ is 3-connected just serves to ensure that $\Om$ has no balanced loop nor balanced 2-cycle.) 
We show that either up to roll-ups $\Om$ uniquely represents $F(\Om)$, or $\Om$ has a collection $H$ of biased subgraphs such that the $H$-reduction of $\Om$ has at most six vertices.  
By Lemma \ref{lem:H_enlargements_are_all}, all representations of $F(\Om)$ are obtained as $H$-enlargements of the biased graphs $\Psi$ with $F(\Psi) = F(\re(\Om,H))$, so in each case finding such an $H$-reduction completes the proof.  

Here an an outline of the proof:  

\begin{itemize} 
\item  If $u$ is the only uncommitted vertex of $\Om$, we show that up to roll-ups $\Om$ uniquely represents $F(\Om)$.   
\item  If $\Om$ has a second uncommitted vertex $v$, then we consider two cases, according to whether $\Om$ has an unbalanced loop incident to $u$, or not.  
\begin{itemize} 
\item  If $\Om$ has an unbalanced loop incident to $u$, we show that there are at most two unbalancing classes in $\delta(u)$ in $\Om-v$.  We then consider two subcases, and find that in each subcase the $H$-reductions of $\Om$ are on at most six vertices.  
\item  If there is no unbalanced loop incident to $u$, we show that there are at most three unbalancing classes in $\delta(u)$ in $\Om-v$.  We consider three subcases, according to the number of unbalancing classes in $\delta(u)$ in $\Om-v$ and in $\Om$.  
We again find that in each subcase there is a collection $H$ of biased subgraphs such that $| V(\re(\Om,H)) | \leq 6$.  
\end{itemize} 
\end{itemize} 

We now proceed with the proof.  

\bigskip 
\noindent \textbf{All but the balancing vertex are committed.} 
If $u$ is the only uncommitted vertex of $\Om$, things are straightforward: 

\begin{thm} 
Let $\Om$ be a biased graph with balancing vertex $u$, and with $F(\Om)$ 3-connected and nongraphic.  
If all vertices $v \in V(\Om) \setminus \{u\}$ are committed, then every biased graph $\Om'$ with $F(\Om') = F(\Om)$ is obtained as a roll-up of $\Om$.  
\end{thm}

\begin{proof} 
Let $A_1, A_2, \ldots, A_k$ be the unbalancing classes of $\delta(u)$.  
Since $F(\Om)$ is nongraphic, $k \geq 3$.  
Since $F(\Om)$ is 3-connected, there is at most one loop $l$ incident to $u$, which is unbalanced.  
Since every vertex but $u$ is committed, every biased graph representing $F(\Om)$ has a biased subgraph isomorphic to $\Om-u$.  
Let $\Om'$ be a biased graph with $F(\Om') = F(\Om)$.  
Then for every vertex $v \in V(\Om-u)$ there is a vertex $v' \in V(\Om')$ with $\delta(v') = \delta(v)$.  
Moreover, each element represented by a $u$-$v$ edge in $A_i$, $i \in \{1,\ldots, k\}$, is represented in $\Om'$ by either an edge incident to $v'$ or an unbalanced loop incident to $v'$.  
Since $F(\Om)$ is nongraphic, every biased graph representing $F(\Om)$ has $|V(\Om)|$ vertices.  
Hence every biased graph representing $F(\Om)$ may be obtained from $G-u$ by adding a vertex $u'$, and adding the edges in $A_1, \ldots, A_k$, and $l$, such that the resulting biased graph has frame matroid isomorphic to $F(\Om)$.  
Again, since every vertex of $\Om$ but $u$ is committed, for each edge $e=uv$ in a set $A_i$, in $\Om'$ one of the endpoints of $e$ is $v'$, and the only choice is whether $e$ has $u'$ as its other endpoint or $e$ is an unbalanced loop incident to $v'$.  

Since $l \notin \delta(v)$ for any $v \not= u$, $l$ cannot be incident to any vertex $v'$ corresponding to a vertex $v \not= u$ in $\Om$, and so must be incident only to $u'$ in $\Om$, and so remains an unbalanced loop in $\Om'$.  
Now suppose an element $e$ represented by an edge $uv$ in $A_i$, for some $i \in \{1,\ldots,k\}$, is represented by an unbalanced loop incident to $v'$ in $\Om'$.  
Let $f=uw$ be an edge in $A_j$, $j \in \{1, \ldots, k\}$.  
There is a $v$-$w$ path $P$ in $\Om-u$, and a corresponding $v'$-$w'$ path $P'$ with $E(P')=E(P)$ in $\Om'$.  
If $j \not= i$, then $E(P) \cup \{e,f\}$ is independent in $F(\Om)$, and so $f$ is not an unbalanced loop in $\Om'$; $f$ is therefore a $u'$-$w'$ edge in $\Om'$.  
If $j = i$, then $E(P) \cup \{e,f\}$ is a circuit of $F(\Om)$, which implies $f$ must be an unbalanced loop incident to $w'$ in $\Om'$.  
\end{proof}

We now proceed with the case that $\Om$ has a second uncommitted vertex.  

\bigskip
\noindent \textbf{$\Om$ has at least $2$ uncommitted vertices.} 
Let $\Om=(G,\Bb)$ be a 3-connected biased graph with a balancing vertex $u$, with $F(\Om)$ nongraphic and 3-connected, 
and with an uncommitted vertex $v \not= u$.  
Set $V=V(G)$ and $E = E(G)$.  

We now have several cases to consider, according to whether or not there is an unbalanced loop at $u$, the number of unbalancing classes in $\Om$ and in $\Om-v$, and their sizes. 
By Lemma \ref{lem:H_enlargements_are_all}, we just need show that in each case there is a collection of biased subgraphs $H$ such that we may apply an $H$-reduction to obtain a biased graph on at most six vertices.  

\subsection{$\Om$ has an unbalanced loop on $u$}  \label{sec:Om_hasanunbalancedlooponu}
\counterwithin{thm}{subsection}
\setcounter{thm}{0}

We first consider the case that there is an unbalanced loop $l$ incident to $u$.  

\begin{lem}
There are at most two unbalancing classes in $\delta(u)$ in $\Om-v$.  
\end{lem}

\begin{proof} 
Suppose for a contradiction that there are three unbalancing classes in $\delta(u)$ in $\Om-v$.  
Since $\Om-v$ is connected, contracting all edges not incident to $u$ then deleting all but one edge in each of three unbalancing classes yields, together with $l$, a biased graph representing $U_{2,4}$.  
Hence $F(\Om-v)$ is nonbinary, and so by Lemma \ref{lem:x_committed_iff_U24} $v$ is committed, a contradiction.  
\end{proof}

Let $\Sigma_1, \Sigma_2, \ldots, \Sigma_k$ be the unbalancing classes of $\delta(u)$ in $\Om$, with $\Sigma_1, \Sigma_2$ the two unbalancing classes of $\delta(u)$ remaining in $\Om-v$ (with possibly one of $\Sigma_1$ or $\Sigma_2$ empty).  
Since $F(\Om)$ is nongraphic, $k \geq 3$ (otherwise $\Om$ is signed graphic by Proposition \ref{prop:If_no_odd_theta}, and so is graphic by Proposition \ref{prop:vertex_splitting_operation}).  
That is, there is at least one $u$-$v$ edge in an unbalancing class other than $\Sigma_1$ or $\Sigma_2$. 
Since $F(\Om)$ is 3-connected, no two $u$-$v$ edges are in the same unbalancing class.  
Let $C$ be the set of edges in $\delta(v) \setminus \delta(u)$, and let $S =\Sigma_3 \cup \cdots \cup \Sigma_k$ be the set of $u$-$v$ edges not in $\Sigma_1$ or $\Sigma_2$.  
Let $Y = \{l\} \cup \{ e : e$ is a $u$-$v$ edge$\}$, and let $X = E \setminus Y$.  
If $X$ is empty, then $F(\Om) = U_{2,m+1}$, where $m$ is the number of $u$-$v$ edges.  
Every representation of $U_{2,m+1}$ is a roll-up of $\Om$, so Theorem \ref{mainthm:main_bal_vertex_rep} holds in this case.  
So assume $X \not= \emptyset$.  
This implies $|V(\Om)|>2$.  
Now if either $\Sigma_1$ or $\Sigma_2$ has no edge with an endpoint different from $v$, then $(X,Y)$ is a 2-separation of $F(\Om)$, a contradiction.  
Hence $X$ contains an edge in each of $\Sigma_1$ and $\Sigma_2$.  
Let $W=V \setminus \{u,v\}$.  
Since $\Om$ is 3-connected, 
the set of neighbours of $u$ in $W$ has size at least 2 
(else the single neighbour together with $v$ would separate $u$ from the rest of $\Om$), 
and the set of neighbours of $v$ in $W$ has size at least 2 
(else the single neighbour together with $u$ would separate $v$ from the rest of $\Om$). 

\begin{lem}\label{lem:Omis2conn} 
$\Om[X]$ is 2-conected. 
\end{lem}

\begin{proof} 
A cut vertex in $\Om[X]$ would imply the existence of a 2-separation of $F(\Om)$, by Lemma \ref{lem:balancedsideof2sepa2sep}. 
\end{proof}

We consider two subcases: since $S$ is nonempty, either $|S| \geq 2$ or $|S|=1$. 

\subsubsection*{Subcase 1. There are at least two $u$-$v$ edges not in $\Sigma_1 \cup \Sigma_2$.}
\mbox{} 
\smallskip

\noindent 
Suppose that $|S| \geq 2$ (Figure \ref{fig:anotherfigureofHes1}).  
\begin{claim} 
Every vertex in $W$ is committed.  
\end{claim} 
\begin{proof}[Proof of Claim] 
Both $u$ and $v$ have at least two neighbours in $W$, and by Lemma \ref{lem:Omis2conn} $\Om$ is 2-connected. 
Hence for every $x \in W$, the biased subgraph $\Om[X]-x$ is connected.   
A $u$-$v$ path in $\Om[X] - x$ together with $l$ and two edges in $S$ yields a $U_{2,4}$ minor in $F(\Om - x)$.  
Thus by Lemma \ref{lem:x_committed_iff_U24} $x$ is committed. 
\end{proof}

Let $H=\Om[X]$.  
Then $H$ is a pinch. 
Since $W \eq V(H) \setminus \{u,v\}$ is nonempty, and all vertices of $H-\{u,v\}$ are committed, we may apply an $H$-reduction.  
The biased graph $\re(\Om,H)$ obtained by replacing $H$ with a pinched triangle $abc$ is the 2-vertex graph shown at right in Figure \ref{fig:anotherfigureofHes1}, so this completes the proof of Theorem \ref{mainthm:main_bal_vertex_rep} in this case. 
\begin{figure}[htbp] 
\begin{center} 
\includegraphics[scale=0.9]{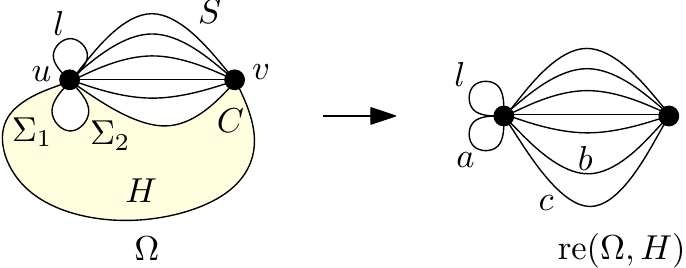}
\end{center} 
\caption{Applying an $H$-reduction in Case 1.} 
\label{fig:anotherfigureofHes1} 
\end{figure}

\subsubsection*{Subcase 2. There is just one $u$-$v$ edge not in $\Sigma_1 \cup \Sigma_2$.}
\mbox{} 
\smallskip

\noindent 
Now suppose $|S|=1$.   
Suppose further that each of $\Sigma_1$ and $\Sigma_2$ have size at least two.  
By Lemma \ref{lem:Omis2conn}, $\Om$ is 2-connected. 
Thus, regardless of the possibilities for the endpoints of edges in $\Sigma_1$ and $\Sigma_2$, for all $x \in W$ there remains a contrabalanced theta in $\Om-x$. 
Together with $l$ such a contrabalanced theta yields a $U_{2,4}$ minor in $F(\Om-x)$.  
Hence we again find, by Lemma \ref{lem:x_committed_iff_U24}, that all vertices in $W$ are committed. 
Again take $H = \Om[X]$.  
As in subcase 1, $H$ is a pinch and applying an $H$-reduction yields a 2-vertex contrabalanced biased graph.   

So suppose now $|\Sigma_1|=1$ while $|\Sigma_2|>1$ (Figure \ref{fig:G-u_bal_unbal_loop_2}).  
\begin{figure}[htbp] 
\begin{center} 
\includegraphics[scale=0.9]{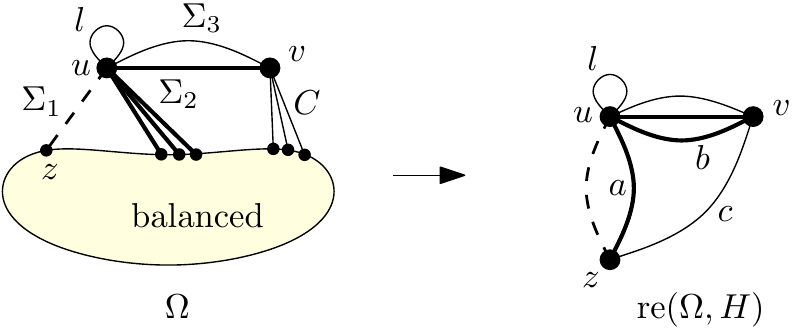}
\end{center} 
\caption[Case (a)ii. $\Om$ and $\re(\Om,H)$]{There is only one edge in $S$ and $|\Sigma_1|=1$ while $|\Sigma_1|>1$.}
\label{fig:G-u_bal_unbal_loop_2} 
\end{figure} 
Let $z \in W$ be the endpoint of the edge in $\Sigma_1$. 
Suppose first $z$ has at least two neighbours in $W$.  
Then for every $x \in W \setminus \{z\}$ there remains an countrabalanced theta in $\Om-x$, which together with $l$ yields a $U_{2,4}$ minor in $F(\Om-x)$. 
Thus every vertex $x \in W \setminus \{z\}$ is committed.  
Let $H$ be the balanced biased subgraph formed by $\Om[W]$ together with the edges in $\Sigma_2 \cap X$ and $C$.  
Replacing $H$ by a balanced triangle $abc$, we obtain a three-vertex biased graph $\re(\Om,H)$ (Figure \ref{fig:G-u_bal_unbal_loop_2} at right).  

Now suppose $z$ has just one neighbour $y$ in $W$.  
Since $F(\Om)$ is 3-connected, there is a second $uz$ edge, which is in $\Sigma_2$, and $y$ has at least two neighbours in $W$ distinct from $z$.  
Neither $y$ nor $z$ is committed. 
However, every vertex $x \in W \setminus \{y,z\}$ is committed, since for every such vertex $x$ there is a contrabalanced theta in $\Om-x$ which together with $l$ forms a $U_{2,4}$-minor. 
Let $H$ be the balanced biased subgraph obtained by deleting from $\Om$ the edges in $Y$ and the three edges incident to $z$.  
Applying an $H$-reduction yields $\re(\Om,H)$, a biased graph with four vertices. 

Finally, suppose there is only one edge in $S$ and that $|\Sigma_1|=|\Sigma_2|=1$ (Figure \ref{fig:G-u_bal_unbal_loop_3}).  
Let $z, w$ be the endpoints in $W$ of the single edge in $\Sigma_1$, $\Sigma_2$, respectively.  
\begin{figure}[htbp] 
\begin{center} 
\includegraphics[scale=0.9]{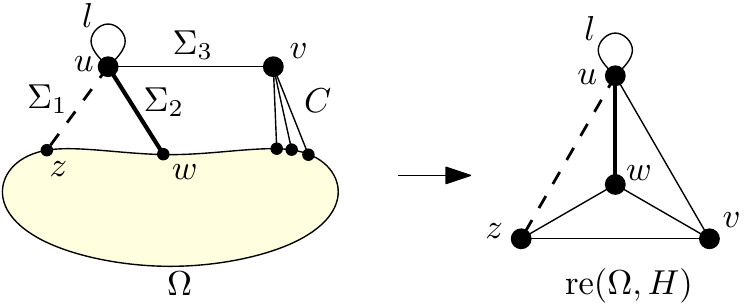}
\end{center} 
\caption{If $|S|=1$ and $|\Sigma_1|=|\Sigma_2|=1$.}
\label{fig:G-u_bal_unbal_loop_3} 
\end{figure} 
We claim that every vertex $x \in W \setminus \{w,z\}$ is committed.  
Connectivity implies that each of $w$ and $z$ have at least two neighbours in $W$.  
By Lemma \ref{lem:Omis2conn} for every $x \in W \setminus \{w,z\}$, $\Om-x$ is connected; moreover $\Om-x$ has an edge in both $\Sigma_1$ and $\Sigma_2$. 
Thus $\Om-x$ contains a countrabalanced theta.  
Together with $l$, this yields a $U_{2,4}$ minor in $F(\Om-x)$, establishing the claim.  
Hence $H=\Om-u$ is a balanced subgraph of $\Om$ having all vertices but $\{v,w,z\}$ committed.  
An $H$-reduction yields the 4-vertex biased graph shown at right in Figure \ref{fig:G-u_bal_unbal_loop_3}.  

This exhausts the possibilities for 3-connected biased graphs with a balancing vertex and an unbalanced loop.  

\subsection{$\Om$ has no unbalanced loop on $u$}

We now consider the case that there is no unbalanced loop incident to $u$. 

First observe that there are at most three unbalancing classes in $\delta(u)$ in $\Om-v$: otherwise since $\Om-v$ is connected, contracting all edges not incident to $u$ then deleting all but one edge in each of four unbalancing classes would yield a biased graph representing $U_{2,4}$, implying $v$ is committed, a contradiction.  

We consider several cases, according to the number of unbalancing classes of $\delta(u)$ in $\Om$ and in $\Om-v$, and their sizes. 
We consider the following three subcases, each of which is dealt with by considering further subcases: 
\begin{enumerate}[1.]
\item $\delta(u)$ has three unbalancing classes in $\Om-v$, and just three unbalancing classes in $\Om$; 
\item $\delta(u)$ has three unbalancing classes in $\Om-v$, and more than three unbalancing classes in $\Om$; 
\item $\delta(u)$ has less than three unbalancing classes in $\Om-v$.  
\end{enumerate} 

\subsubsection*{Subcase 1. $\delta(u)$ has 3 unbalancing classes in $\Om-v$, and just 3 unbalancing classes in $\Om$} \label{sec:3sim3sim} 
\mbox{} 
\smallskip 

\noindent 
A \emph{fat theta} is a biased graph that is the union of three subgraphs $A_1, A_2, A_3$ mutually meeting at just a single pair of vertices, in which a cycle $C$ is balanced \iiff $C \subseteq A_i$ for some $i \in \{1,2,3\}$ (Figure \ref{fig:A_fat_theta}).  
\begin{figure}[tbp] 
\begin{center} 
\includegraphics[scale=0.9]{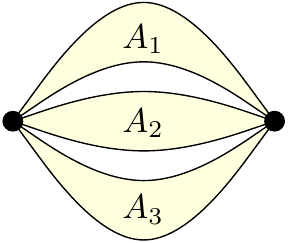}
\end{center} 
\caption{A fat theta.}
\label{fig:A_fat_theta} 
\end{figure} 

\begin{lem} \label{lem:G-v_isafattheta}
If there are three unbalancing classes of $\delta(u)$ in $\Om-v$, then $\Om-v$ is a fat theta.  
\end{lem} 

\begin{proof} 
This follows immediately from Lemmas \ref{lem:x_committed_iff_U24} and \ref{lem:shortcut_an_odd_theta_U_24}.   
\end{proof} 

\begin{lem} \label{lem:3simclassesatmostoneofsizeone}
At most one unbalancing class of $\delta(u)$ in $\Om$ has size one.  
\end{lem}

\begin{proof} 
Suppose for a contradiction that there are two unbalancing classes in $\Om$ of size one, say $\Sigma_1$ and $\Sigma_2$, with edge $a \in \Sigma_1$ and edge $b \in \Sigma_2$.  
Since $a$ is not in any balanced cycle, every circuit of $F(\Om)$ containing $a$ is either a countrabalanced theta or a pair of handcuffs.  
A countrabalanced theta must contain an edge from each of the three unbalancing classes, and so contains $b$.  
A pair of handcuffs contain two unbalanced cycles meeting at $u$.  
Since there are only three unbalancing classes, every pair of handcuffs contains both $a$ and $b$.  
Hence every circuit containing $a$ contains $b$.  
Similarly, every circuit containing $b$ contains $a$.  
Hence $\{a,b\}$ is a cocircuit of size 2.  
This contradicts the fact that $F(\Om)$ is 3-connected.  
\end{proof}

We consider two further subcases, according to whether or not $\Om$ has a unbalancing class of size one.  

\bigskip
\paragraph{\emph{1A.  $\Om$ has an unbalancing class of size 1}}

Suppose there is an unbalancing class of $\delta(u)$ of size one.  
If $|V(\Om)| \leq 3$ then we are done, 
so assume $|V(\Om)| \geq 4$. 
By Lemmas \ref{lem:G-v_isafattheta} and \ref{lem:3simclassesatmostoneofsizeone} $\Om$ has the form of one of the biased graphs shown in Figure \ref{fig:GB_with_bal_vertex_3simclasses}, 
where each of the biased subgraphs $H_1, H_2, H_3$ are balanced and connected.  
\begin{figure}[tbp] 
\begin{center} 
\includegraphics[scale=0.9]{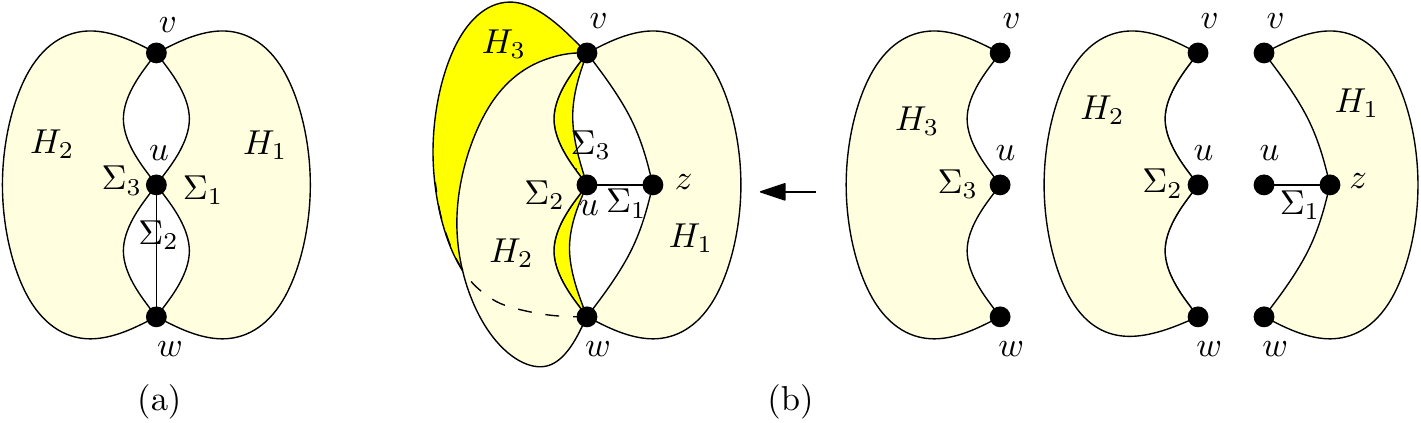}
\end{center} 
\caption[Case (b)i. A. $\Om$ has a unbalancing class of size one]{If $\Om$ has a unbalancing class of size one.  Biased graph (b) is obtained by identifying the vertices labelled $v$, those labelled $u$, and those labelled $w$ in each of graphs $H_1$, $H_2$, $H_3$ at right.}
\label{fig:GB_with_bal_vertex_3simclasses} 
\end{figure} 
Our next lemma tells us precisely which vertices of $\Om$ are committed in this case.   

\begin{lem}
Let $\Om$ be a 3-connected biased graph with $F(\Om)$ nongraphic and 3-connected, with a balancing vertex $u$ and a second uncommitted vertex $v \not= u$, with no loop incident to $u$.  Suppose there are exactly three unbalancing classes in $\Om$ and in $\Om-v$, and that $\Om$ has a unbalancing class of size one.  Then $\Om$ has the form of one of the biased graphs shown in Figure \ref{fig:GB_with_bal_vertex_3simclasses_uncommitted_vertices}, where all vertices $t \in V(\Om) \setminus \{u,v,w,x,y,z\}$ 
are committed.  
\end{lem} 

\begin{figure}[htbp] 
\begin{center} 
\includegraphics[scale=0.9]{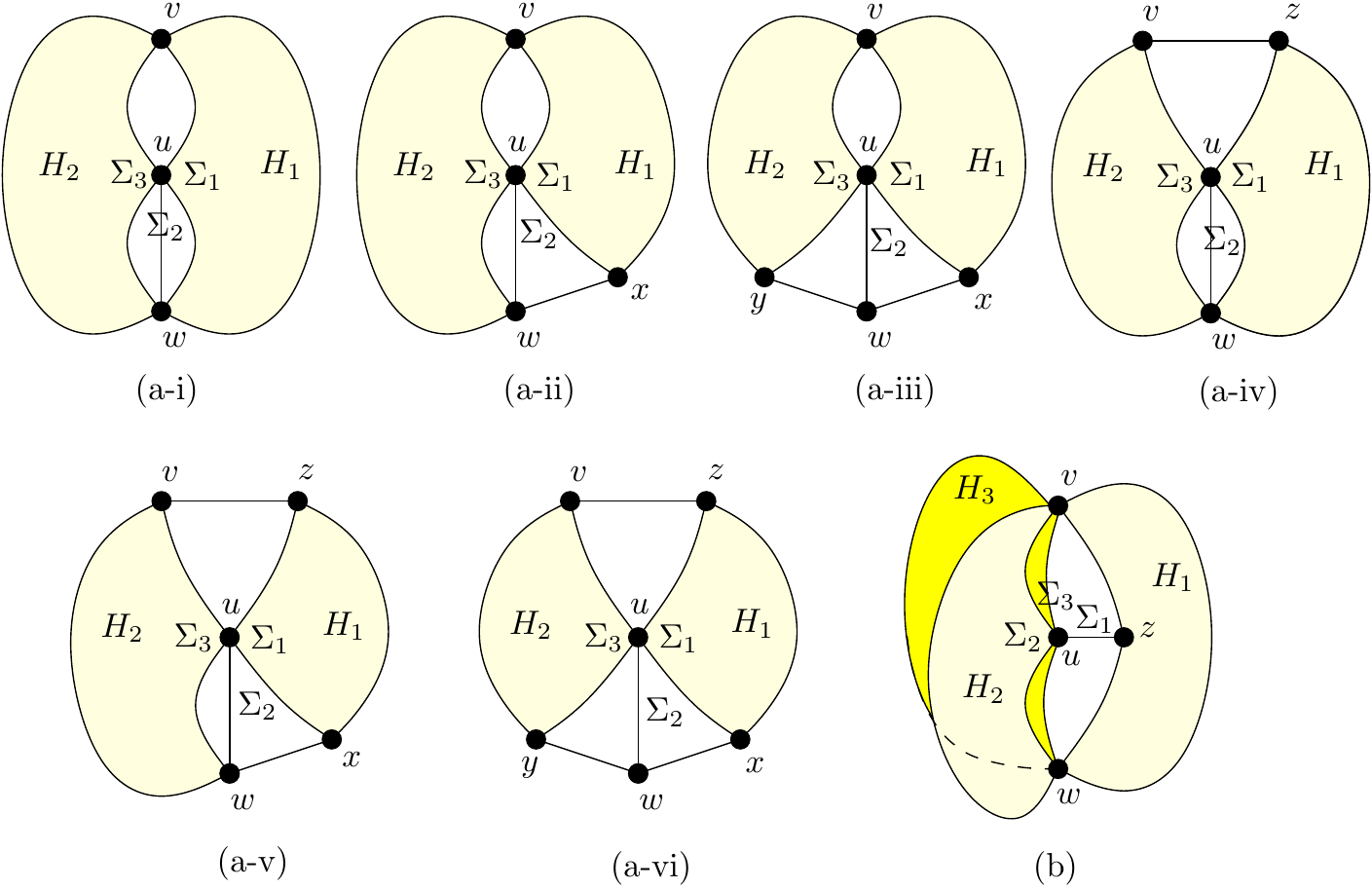}
\end{center} 
\caption[Case (b)i. A. Lobes when $\Om$ has a unbalancing class of size one]{Possibilities for lobes in $\Om$ in the case $\Om$ has a unbalancing class of size one.}
\label{fig:GB_with_bal_vertex_3simclasses_uncommitted_vertices} 
\end{figure} 

\begin{proof} 
Suppose first $\Om$ has the form of biased graph (a) in Figure \ref{fig:GB_with_bal_vertex_3simclasses}.  
Let $x \in V(\Om) \setminus \{u,v,w\}$.  
As long as in $\Om-x$ there are $u$-$w$ paths $P \subseteq H_1$ and $P' \subseteq H_2$, and $u$-$v$ paths $Q \subseteq H_1$ and $Q' \subseteq H_2$, there is a $U_{2,4}$ minor in $F(\Om-x)$, so $x$ is committed.  
Suppose $x \in V(H_1)$, say, is not committed.  
Then the deletion of $x$ must destroy either all $u$-$w$ paths in $H_1$ or all $u$-$v$ paths in $H_1$.  
That is, $x$ is a cut vertex in $H_1$.  
Connectivity (via Lemma \ref{lem:balancedsideof2sepa2sep}) implies now that either $x$ is incident to $v$ and there are no other vertices in $H_1$ incident to $v$, or that $x$ is incident to $w$ and there are no other vertices in $H_1$ incident to $w$.  
In other words, $\Om$ has the form of one of biased graphs (a-i)-(a-vi) in Figure \ref{fig:GB_with_bal_vertex_3simclasses_uncommitted_vertices}.  

Now suppose $\Om$ has the form of biased graph (b) in Figure \ref{fig:GB_with_bal_vertex_3simclasses}.  
Let $x \in V(H_1) \setminus \{u,v,w\}$.  
As long as in $\Om-x$ there is either a $u$-$w$ or a $u$-$v$ path contained in $H_1$, there is a $U_{2,4}$ minor in $F(\Om-x)$ so $x$ is committed.  
Hence if $x \in V(H_1)$ is not committed, the deletion of $x$ must destroy all such paths.  
Connectivity implies then that $x = z$.  
So suppose now $x \in V(H_2) \setminus \{u,v,w\}$.  
Again, as long as there is either a $u$-$v$ or a $u$-$w$ path in $H_2$ avoiding $x$, $x$ is committed.  
Hence if $x \in V(H_2)$ is not committed, there are no such paths in $H_2$ avoiding $x$.  
Connectivity now implies that $x$ is incident to $u$ and that there are no other vertices in $H_2$ incident to $u$.  
But this is a contradiction, as then both the $\Sigma_1$ and $\Sigma_2$ unbalancing classes of $\delta(u)$ are of size 1, contradicting Lemma \ref{lem:3simclassesatmostoneofsizeone}.  
Similarly, every vertex $x \in V(H_3) \setminus \{u,v,w\}$ is committed. 
\end{proof}

In any case, taking $H = \{H_1, H_2, H_3\}$ and applying an $H$-reduction, we obtain a biased graph $\re(\Om,H)$ on at most 6 vertices, completing the proof in this case.  

\bigskip  
\paragraph{\emph{1B.  Each unbalancing class has size greater than $1$}}
Now suppose that in $\Om$ each unbalancing class of $\delta(u)$ has size greater than one.  
Let $w$ be the second balancing vertex of the fat theta $\Om-v$.  
Together with their edges incident to $v$, the three balanced subgraphs of the fat theta $\Om-v$ are naturally extended to three lobes $H_1$, $H_2$, $H_3$ of $\Om$, which meet at $\{v,u,w\}$.  
Let $H = \{H_1, H_2, H_3\}$.  

\begin{claim} 
Each vertex $x \in V(H) \setminus \{u,v,w\}$ is committed.  
\end{claim}

\begin{proof}
Let $x \in V(\Om) \setminus \{u,v,w\}$; say $x \in V(H_1)$.  
We claim that in $H_1 - x$ there is either a $u$-$v$ path avoiding $w$ or a $u$-$w$ path avoiding $v$.  
For suppose not: then $x$ is a cut vertex of $H_1$ separating $u$ from $\{v,w\}$.  
Since $u$ has at least two neighbours in $H_1$, $\{u,x\}$ determines a 2-separation of $G$, a contradiction.  
So suppose $P$ is a $u$-$w$ path in $H_1-x$ avoiding $v$.  
Since $v$ is not a cut vertex of $H_2$ or $H_3$, there are $u$-$w$ paths $P'$ and $P''$ avoiding $v$ in $H_2$ and $H_3$, respectively.  
Let $Q'$ be a $P'$-$v$ path in $H_2-w$, and $Q''$ be a $P''$-$v$ path in $H_3-w$ (such paths exist, since $w$ is not a cut vertex of $H_2$ nor $H_3$).  
Contracting all edges of $P, P'$, and $P''$ but those incident to $w$, and all edges of $Q'$, and all edges of $Q''$ but its edge incident to $v$ yields a biased graph representing $U_{2,4}$ as a minor of $\Om-x$.  
\end{proof}

Applying an $H$-reduction yields a biased graph on three vertices.  

\subsubsection*{Subcase 2. $\delta(u)$ has 3 unbalancing classes in $\Om-v$ and greater than $3$ unbalancing  classes in $\Om$} \label{sec:3sim4sim} 
\mbox{} 
\smallskip 

\noindent To aid the analysis, we now slightly generalise our concept of a lobe: the \emph{lobes} of $\Om$ are the three balanced biased subgraphs $H_1, H_2, H_3$ of $\Om$ meeting at $\{u,v,w\}$, each of which is obtained from one of the three balanced subgraphs $A_1, A_2, A_3$ whose union is the fat theta $\Om-v$, by adding all edges linking $v$ and a vertex in $A_i$ ($i \in \{1,2,3\}$).  
By the theta property, each $H_i$ must indeed be balanced. 
Call a lobe \emph{degenerate} if it contains only 1 edge.  
If all three lobes are degenerate then $|V(\Om)|=3$ and we are done. So assume $\Om$ has at least one non-degenerate lobe. 
The following four further subcases exhaust the possibilities in the case $\delta(u)$ has three unbalancing classes in $\Om-v$ and more than three unbalancing classes in $\Om$:  
\begin{itemize} 
\item[2A.]  Just 4 unbalancing classes in $\delta(u)$, no degenerate lobes; 
\item[2B.]  Just 4 unbalancing classes in $\delta(u)$, exactly two degenerate lobes; 
\item[2C.]  Just 4 unbalancing classes in $\delta(u)$, exactly one degenerate lobe; 
\item[2D.]  More than 4 unbalancing classes in $\delta(u)$.  
\end{itemize} 
In any case the fact that $F(\Om)$ is 3-connected forces a unbalancing class present in $\Om$ but not $\Om-v$ to be of size 1.  

\bigskip  
\paragraph{\emph{2A. $\Om$ has exactly 4 unbalancing classes, no degenerate lobes}} \label{section:G9}

Suppose there are exactly four unbalancing classes in $\delta(u)$ in $\Om$. 
By Lemma \ref{lem:G-v_isafattheta} and the theta property, $\Om$ has the form shown shown at left in Figure \ref{fig:G-u_bal_3classes_4classes_in_G}, where 
for each $i \in \{1,2,3\}$, all edges in lobe $H_i$ incident to $u$ are in unbalancing class $\Sigma_i$, and 
$\Sigma_4$ is the unbalancing class of $\delta(u)$ not present in $\Om-v$ consisting of a single edge. 
\begin{figure}[htbp] 
\begin{center} 
\includegraphics[scale=.9]{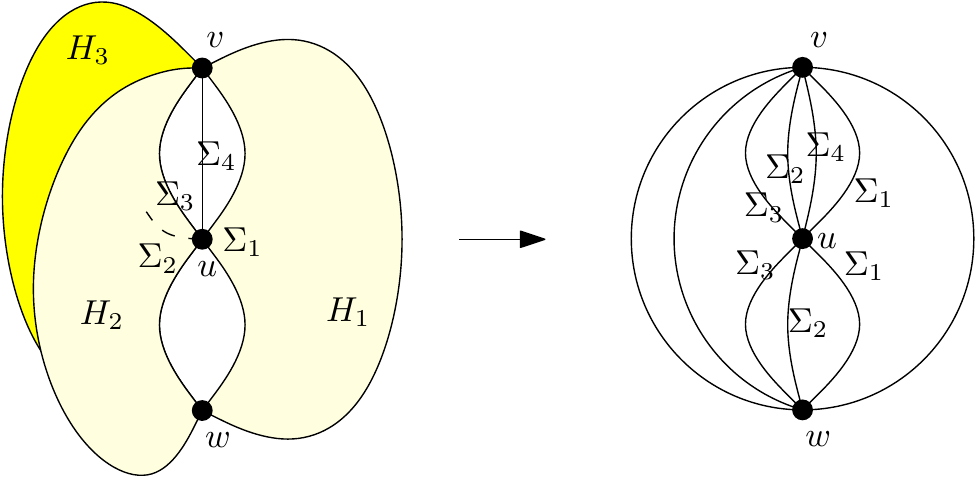}
\end{center} 
\caption{If all vertices but $u, v$ are committed: $\Om$ and $\re(\Om,H)$}
\label{fig:G-u_bal_3classes_4classes_in_G} 
\end{figure} 

\begin{claim} 
Every vertex but $u$ and $v$ is committed. 
\end{claim}

\begin{proof} 
Connectivity implies $w$ is not a cut vertex in any of $H_1, H_2, H_3$. 
Hence there are $u$-$v$ paths in each of $H_1, H_2, H_3$ avoiding $w$.  
Together with the $u$-$v$ edge in $\Sigma_4$, these yield a $U_{2,4}$ minor, so $w$ is committed.  
Now consider $x \in V(\Om) \setminus \{u,v,w\}$.  
Suppose \wolog{} $x \in H_1$.  
Choose $u$-$v$ paths $P \subseteq H_2$, $P' \subseteq H_3$ avoiding $w$, a $P$-$w$ path $Q \subseteq H_2$ avoiding $v$, and a $P'-w$ path $Q' \subseteq H_3$ avoiding $v$.  
Together with the $uv$ edge in $\Sigma_4$, these yield a $U_{2,4}$ minor in $F(\Om-x)$, so $x$ is committed.  
\end{proof}

Taking $H = \{H_1, H_2, H_3\}$, an $H$-reduction yields a biased graph on three vertices (Figure \ref{fig:G-u_bal_3classes_4classes_in_G}).  

\bigskip 
\paragraph{\emph{{2B. $\Om$ has exactly 4 unbalancing classes and exactly 2  degenerate lobes}}}
Suppose there is just one edge in $\Sigma_2$ and just one edge in $\Sigma_3$.  
Let $x \in V(\Om) \setminus \{u,v,w\}$.  
Then $x$ is committed unless the deletion of $x$ destroys either all $u$-$w$ paths or all $u$-$v$ paths.  
Hence connectivity implies $\Om$ has the form of one of the biased graphs shown in Figure \ref{fig:G-u_bal_4classes_degenerate_lobes}, where $|E(H_1)| \geq 3$ and every vertex $z \in V(\Om) \setminus \{u,v,w,x,y\}$ is committed.  
\begin{figure}[htbp] 
\begin{center} 
\includegraphics[scale=.9]{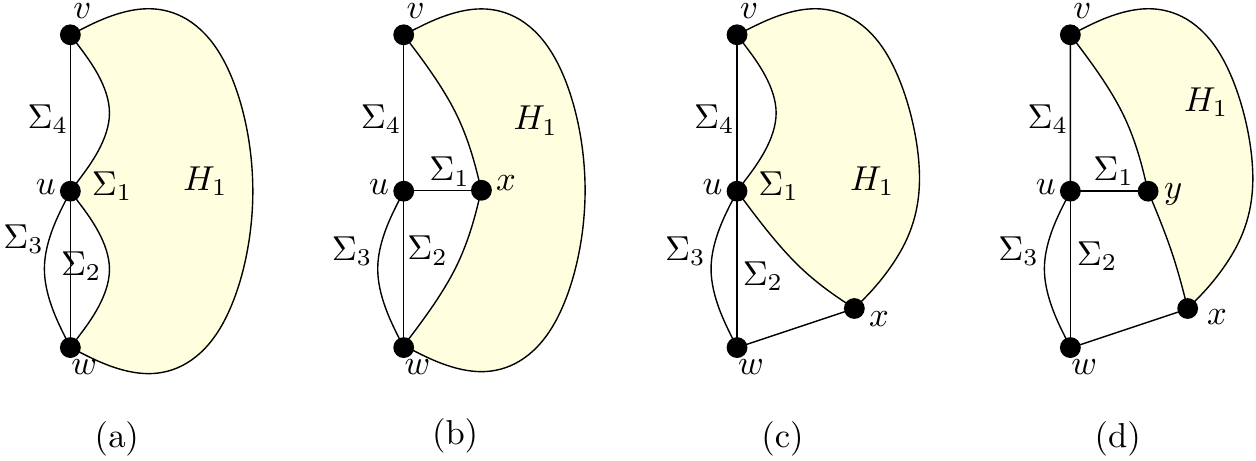}
\end{center} 
\caption[Case (b)ii.B]{Case 2B. All vertices except $u,v,w,x,y$ are committed.}
\label{fig:G-u_bal_4classes_degenerate_lobes} 
\end{figure} 
Applying an $H_1$-reduction, we obtain a biased graph on at most five vertices.

\smallskip 
\paragraph{\emph{2C.  $\Om$ has exactly 4 unbalancing classes and exactly 1 degenerate lobe}}
In this case $\Om$ is as shown in Figure \ref{fig:G-u_bal_4classes_1_degenerate_lobe}(a).  
\begin{figure}[htbp] 
\begin{center} 
\includegraphics[scale=.9]{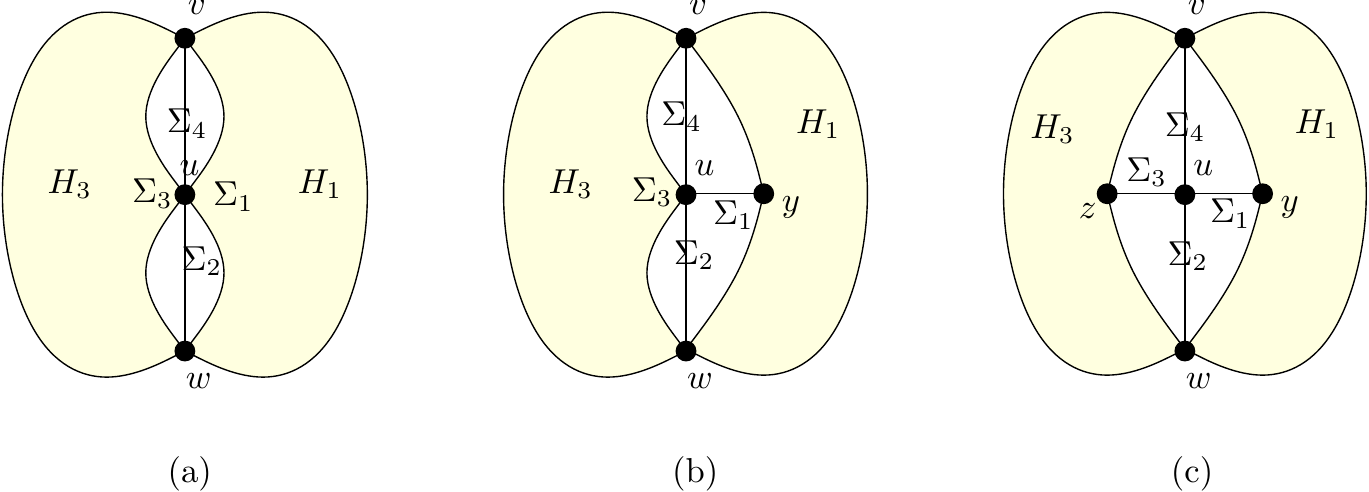}
\end{center} 
\caption[Case (b)ii.C]{Case 2C. $\Om$ has exactly four balancing glasses and exactly one degenerate lobe.}
\label{fig:G-u_bal_4classes_1_degenerate_lobe} 
\end{figure} 
Let $x \in V(\Om) \setminus \{u,v,w\}$.  
It is easy to see that $x$ is committed unless the deletion of $x$ destroys both all $u$-$v$ and all $u$-$w$ paths.  
Hence $\Om$ has the form of biased graph (a), (b), or (c) of Figure \ref{fig:G-u_bal_4classes_1_degenerate_lobe}, where all vertices $x \notin \{u,v,w,y,z\}$ are committed.  
Applying an $\{H_1,H_2\}$-reduction replaces each lobe $H_1$, $H_2$ with a balanced triangle, and we obtain a biased graph on at most five vertices.

\bigskip
\paragraph{\emph{2D. $\Om$ has $>4$ unbalancing classes}}
Suppose now $\Om$ has more than four unbalancing classes in $\delta(u)$.  
Then $\Om$ has the form shown at left in Figure \ref{fig:G-u_bal_3classes_5classes_in_G}.  
\begin{figure}[htbp] 
\begin{center} 
\includegraphics[scale=0.9]{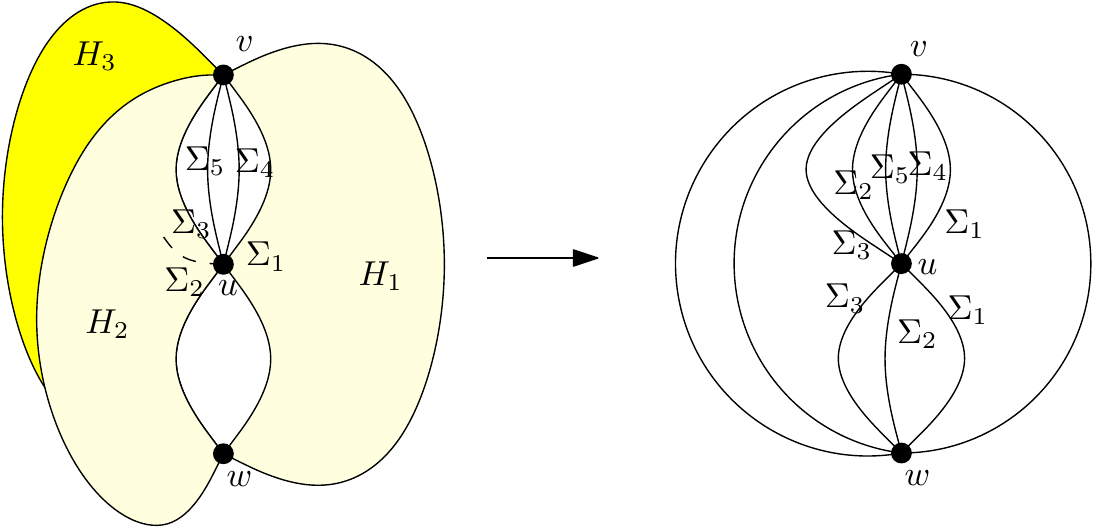}
\end{center} 
\caption[Case (b)ii.D $\re(\Om,H)=G_{17}$]{Case 2D. $\Om$ and $\re(\Om,H)$.}
\label{fig:G-u_bal_3classes_5classes_in_G} 
\end{figure} 
It may be that $\Om$ has none, one, or two degenerate lobes. 

Suppose first that none of the lobes $H_1$, $H_2$, $H_3$ is degenerate.  
Then there is a $u$-$v$ path avoiding $w$ in each of $H_1$ and $H_2$, and so $F(G-w)$ has a $U_{2,4}$ minor, so $w$ is committed.  
Let $x \in V(\Om) \setminus \{u,v,w\}$.  Since in each of the two lobes not containing $x$ there is both a $u$-$v$ path and a $u$-$w$ path, we find a $U_{2,4}$ minor in $F(\Om-x)$, so $x$ is committed.  
Applying an $\{H_1, H_2, H_3\}$-reduction, we obtain a biased graph on three vertices.  

Suppose now $\Om$ has one degenerate lobe.  
Suppose lobe $H_2$ has size one, both $|H_1|, |H_3| > 1$.  
Let $x \in V(\Om) \setminus \{u,v\}$.  
If $x = w$, choose in $G-w$ a $u$-$v$ path in $H_1$ and a $u$-$v$ path in $H_3$: we find a $U_{2,4}$ minor in $F(G-w)$.  
If $x \not= w$, choose a $u$-$v$ and a $u$-$w$ path in the lobe not containing $x$: we thus find a $U_{2,4}$ minor in $F(\Om-x)$.  
Thus every vertex $x \in V(\Om) \setminus \{u,v\}$ is committed; the biased graph $\re(\{H_1, H_3\},\Om)$ has three vertices.  

Suppose now $\Om$ has two degenerate lobes, $H_2$ and $H_3$, with $|H_1| > 1$ (Figure \ref{fig:G-u_bal_5classes_2_degen_lobes}).  
\begin{figure}[htbp] 
\begin{center} 
\includegraphics[scale=0.9]{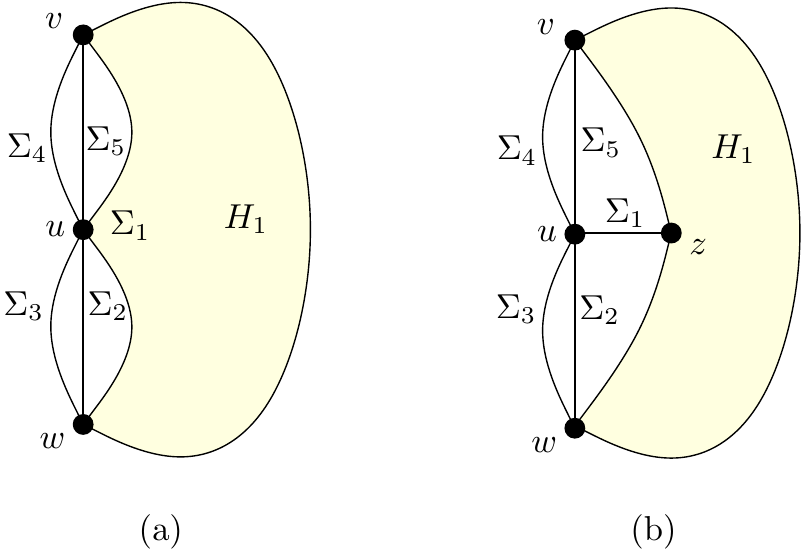}
\end{center} 
\caption[Case (b)ii.D]{$\Om$ has two degenerate lobes.}
\label{fig:G-u_bal_5classes_2_degen_lobes} 
\end{figure} 
Let $x \in V(\Om) \setminus \{u,v,w\}$.  
Unless the deletion of $x$ destroys both all $u$-$v$ and all $u$-$w$ paths, there is a $U_{2,4}$ minor in $F(\Om-x)$.  
Hence $\Om$ has the form of one of biased graphs (a) or (b) in Figure \ref{fig:G-u_bal_5classes_2_degen_lobes} where all vertices $x \in V(H_1) \setminus \{u,v,w,x\}$ are committed.  
The biased graph $\re(H_1,\Om)$ has at most four vertices.

\subsubsection*{Subcase 3. $\delta(u)$ has less than $3$ unbalancing classes in $\Om-v$}  \label{sec:2sim} 
\mbox{}
\smallskip

\noindent 
We may assume that $\Om$ does not have an uncommitted vertex $z$ leaving three unbalancing classes in $\delta(u)$ in $G-z$, since we have dealt with this case in the previous section.  
We consider two further subcases, depending on the number of unbalancing classes in $\delta(u)$ in $\Om-v$.

\bigskip 
\paragraph{\emph{3A. There is just 1 unbalancing class of $\delta(u)$ in $\Om-v$}}
In this case, $v$ is also balancing.  
Moreover, there must be at least four unbalancing classes in $\Om$, else $F(\Om)$ would be graphic.  
Hence there are at least three $u$-$v$ edges each of which is contained in a distinct unbalancing classes.  
But taking $X$ to be the set of edges not linking $u$ and $v$, the partition $(X, E \setminus X)$ is a 2-separation of $F(\Om)$ by Lemma \ref{lem:balancedsideof2sepa2sep}, a contradiction.  
Hence this case does not occur.   

\bigskip 
\paragraph{\emph{3B. There are 2 unbalancing class in $\Om-v$.}  \label{sec:2classes_in_G-v}}
Let $\Sigma_1, \Sigma_2$ be the unbalancing classes of $\delta(u)$ in $\Om-v$.  
Let $S = \delta(u) \cap \delta(v) \setminus \br{\Sigma_1 \cup \Sigma_2}$.  
Since $F(\Om)$ is nongraphic, there is at least one $u$-$v$ edge in a unbalancing class distinct from $\Sigma_1$ and $\Sigma_2$, so $S$ is nonempty (Figure \ref{fig:anotherfigureofHes22}).
Let $X = E \setminus \{ e : e$ is a $uv$ edge$\}$.  
Let $H = \Om[X]$.  
Let $W = V \setminus \{u,v\}$. 

\begin{figure}[htbp] 
\begin{center} 
\includegraphics[scale=.9]{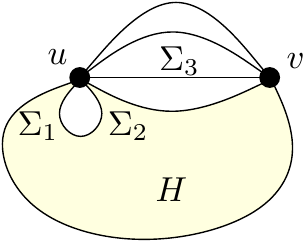}
\end{center} 
\caption{Case 3B: Two unbalancing classes in $\Om-v$.}
\label{fig:anotherfigureofHes22} 
\end{figure} 
\begin{claim} 
Every vertex $x \in W$ is committed.  
\end{claim}

\begin{proof}[Proof of Claim] 
If $|S| \geq 3$, then since there is a $u$-$v$ path in $H-x$ for every $x \in W$, we easily find a $U_{2,4}$ minor in $F(\Om - x)$.  
If $|S| = 2$, and both $|\Sigma_1|$ and $|\Sigma_2|$ are at least two, then again, for every $x \in W$ there is a $u$-$v$ path through $H - x$, and so a $U_{2,4}$ minor in $F(\Om-x)$.  
Hence every $x \in W$ is committed in these cases.  

So suppose $|S|=2$ and one of $\Sigma_1$ or $\Sigma_2$ has size one, say $|\Sigma_1|=1$.  
Let $z$ to be the endpoint in $W$ of the edge in $\Sigma_1$.  
Then $\Om-z$ has three unbalancing classes remaining in $\delta(u)$, and $\Om-z$ is a fat theta so $z$ is not committed.  
This contradicts our assumption that no such vertex exists (having already dealt with this case in the previous section).  

So suppose now $|S|=1$.  
If both $|\Sigma_1|$ and $|\Sigma_2|$ are at least two, then for any $x \in W$, there are three unbalancing classes in $\Om-x$, so it must be that $x$ is committed (else we again contradict our assumption that there is no uncommitted vertex whose deletion leaves three unbalancing classes).  
So finally suppose one of $\Sigma_1$ or $\Sigma_2$ has size one, say, \wolog, $|\Sigma_1|=1$.  
Then the edge in $\Sigma_1$ is in series with the edge in $S$, contradicting the fact the $F(\Om)$ is 3-connected.  
\end{proof}

Since $H$ is a pinch with signature contained in $\delta(u)$ and $H$ has every vertex aside from $u$ and $v$ committed, we may apply an $H$-reduction. 
The biased graph $\re(H,\Om)$ has two vertices.  
This completes the proof of Theorem \ref{mainthm:main_bal_vertex_rep}.  
\qed

\section*{Acknowledgement}

We wish to express our thanks to the two anonymous referees for their careful reading of the manuscript. Their thoughtful suggestions have resulted in a much improved paper.

\bibliography{Matroids1.bib} 
\bibliographystyle{plain} 

\end{document}